\documentclass[11pt,reqno,oneside]{amsart}
\usepackage{amsfonts,amssymb,amsmath,epic,eepic,float,fullpage}
\usepackage{longtable}
\usepackage[usenames,dvipsnames]{color}
\usepackage{siunitx}
\usepackage{empheq}
\usepackage[mathscr]{euscript}
\usepackage[T1]{fontenc}
\usepackage{mathtools}
\usepackage{comment}
\usepackage{bm}
\usepackage{caption}
\usepackage{hyperref}
\usepackage{enumerate,enumitem}
\usepackage{tikz}
\usetikzlibrary{graphs,patterns,decorations.markings,arrows,matrix}
\usetikzlibrary{calc,decorations.pathmorphing,decorations.pathreplacing,shapes}
\allowdisplaybreaks

\newtheorem{thm}{Theorem}[section]
\newtheorem{lem}[thm]{Lemma}
\newtheorem{prop}[thm]{Proposition}

\newtheorem{cor}[thm]{Corollary}

\theoremstyle{definition}
\newtheorem{defn}[thm]{Definition}

\theoremstyle{remark}
\newtheorem{rmk}[thm]{Remark}

\theoremstyle{remark}
\newtheorem{ex}[thm]{Example}
\long\def\junk#1{}

\def\leq{\leqslant}
\def\geq{\geqslant}

\def\buv{(\mathbf{u},\mathbf{v})}

\colorlet{lgray}{white!85!black}
\colorlet{lblue}{white!85!blue}
\colorlet{lblack}{white!85!black}
\colorlet{lred}{white!85!red}
\colorlet{lgreen}{white!80!green}
\colorlet{dgreen}{black!30!green}
\definecolor{green}{rgb}{0.1,0.8,0.1}
\definecolor{yellow}{rgb}{1.0,0.85,0.25}
\usepackage{ytableau}

\renewcommand\le{\leq}
\renewcommand\ge{\geq}

\newcommand{\ff}{\mathbb{F}}

\numberwithin{equation}{section}

\title{Inhomogeneous $q$-Whittaker polynomials II: ring theorem and positive specializations}
\author{Ajeeth Gunna \and Damir Yeliussizov}

\address{Ajeeth Gunna, University of Melbourne, Australia.}
\email{ajeeth.gunna@unimelb.edu.au}
\address{Damir Yeliussizov, Kazakh-British Technical University, Almaty, Kazakhstan.}
\email{yeldamir@gmail.com}

\begin{document}
\begin{abstract}
We study inhomogeneous $q$-Whittaker polynomials which extend both $q$-Whittaker and stable Grothendieck polynomials. We prove that inhomogeneous $q$-Whittaker polynomials (in countably many variables) form a basis of certain commutative ring extending the ring of symmetric functions to a subring of its completion.  
We then describe positive specializations of that ring and relate them with a subset of Macdonald-positive specializations of the ring of symmetric functions. 
We also show some related probability distributions obtained from  positive specializations of inhomogeneous $q$-Whittaker polynomials. 
\end{abstract}

\maketitle
\setcounter{tocdepth}{1}
\makeatletter
\def\l@subsection{\@tocline{2}{0pt}{2.5pc}{5pc}{}}
\makeatother
\tableofcontents
\section{Introduction}

\subsection{Motivation and context}
In this paper, we study some structural questions about a family of symmetric functions called {\it inhomogeneous $q$-Whittaker polynomials}. These functions are a special case of a broader class of symmetric polynomials introduced by Borodin and Korotkikh \cite{spin-BK2024} known as  {\it inhomogeneous spin $q$-Whittaker polynomials}. 

These functions arise from quantum integrable systems, in the form of two-dimensional exactly solvable lattice models, as partition functions. Lattice models have been systematically employed in the study of symmetric functions. In this approach, the symmetric polynomials are expressed as partition functions and then by leveraging the underlying {Yang--Baxter equation} one can derive a range of properties and identities with conceptually evident proofs. Furthermore, 
this approach allows to construct new highly nontrivial families of symmetric polynomials through modifications of local vertex weights. In particular, the additional parameters appearing in the inhomogeneous spin $q$-Whittaker polynomials arise from such modifications to the vertex weights of the lattice model used for {\it spin $q$-Whittaker polynomials} by Borodin and Wheeler in~\cite{spin-BW2021}. We refer the reader to the first instalment of this series~\cite{GWZJ2025} for a comprehensive account of the interplay between lattice models and symmetric polynomials.

Our focus is on inhomogeneous $q$-Whittaker polynomials which exemplify a remarkable feature of lattice models: their ability to unify seemingly unrelated families of symmetric functions. These polynomials extend both $q$-Whittaker polynomials and stable Grothendieck polynomials. While $q$-Whittaker polynomials arise as polynomial eigenfunctions of the $q$-deformed Toda chain \cite{Ruijsenaars1990,Etingof1999} and as characters of certain subspaces of integrable representations of affine algebras \cite{ChariLoktev2006}, stable Grothendieck polynomials appear as polynomial representatives of Schubert varieties in the $K$-theory of Grassmannians \cite{Gcom:B2002}. We expect that the family of inhomogeneous $q$-Whittaker polynomials 
reveals deep and unexpected connections between the diverse contexts in which the two different subclasses of polynomials arise.

\subsection{Summary of main results}
Our main results are the following: 

\medskip

{\bf Basis phenomenon: Ring theorem.} We prove that inhomogeneous $q$-Whittaker functions (in infinitely many variables) form a basis of certain commutative ring extending the ring of symmetric functions to a subring of its completion (by degree filtration), i.e. they have finite product expansions; moreover, this ring extends to a bialgebra via skew versions. 
This is a rather rare and remarkable phenomenon for classes of symmetric functions (with infinitely growing degrees, living in the completion of the ring of symmetric functions) and it generalizes the result of Buch \cite{Gcom:B2002} about stable Grothendieck polynomials. 

\medskip

{\bf Positive specializations.} The ring theorem provides foundation for studying positive specializations on the basis of inhomogeneous $q$-Whittaker functions. We describe real homomorphisms of that ring that are nonnegative on the basis of inhomogeneous $q$-Whittaker functions and their skew versions. This description is related to similar positivity questions about $q$-Whittaker polynomials which are a special case of Macdonald-positive specializations of the ring of symmetric functions,
classified by Matveev who proved the Kerov conjecture \cite{positive:M2019Kerov}. 
The latter generalizes classical Edrei--Thoma theorem on the description of Schur-positive specializations, which is a fundamental result related to total positivity and asymptotic representation theory, see e.g. \cite{Borodin_Olshanski_2016} and references therein. We also show some related probability distributions on integer partitions obtained from such positive specializations of inhomogeneous $q$-Whittaker polynomials. 

\subsection{Inhomogeneous $q$-Whittaker polynomials}

We introduce the family $\{\ff_{\lambda/\mu}\}$ of skew inhomogeneous $q$-Whittaker polynomials through their branching formula, which offers the most direct definition for our purposes:

For two partitions $\lambda$ and $\mu$, the one-variable \emph{skew inhomogeneous $q$-Whittaker polynomial} is defined by
\begin{equation*}
 \mathbb{F}_{\lambda/\mu}(x) := \mathbf{1}_{\mu \prec \lambda}\,\,
 x^{|\lambda|-|\mu|}\,
 \prod_{r=1}^{\ell(\lambda)}
 (x;q)_{\mu_{r}-\lambda_{r+1}}\,
 \frac{(q;q)_{\lambda_{r}-\lambda_{r+1}}}{(q;q)_{\lambda_{r}-\mu_{r}}\,(q;q)_{\mu_{r}-\lambda_{r+1}}}.
\end{equation*}
Then for many variables the inhomogeneous $q$-Whittaker polynomials can be defined recursively via the branching formula: 
$$
\ff_{\lambda/\mu}(x_1, \ldots, x_n) = \sum_{\nu} \ff_{\lambda/\nu}(x_1, \ldots, x_{n-1})\, \ff_{\nu/\mu}(x_n). 
$$

These polynomials are inhomogeneous in terms of their degree and their lowest-degree component coincides with the corresponding skew $q$-Whittaker polynomial. 
Their highest degree increases with the number of variables, as illustrated by the following example:
\[
\ff_{1/1}(x_1,\ldots, x_n) = 1 - \mathbb{F}_{1}(x_{1},\dots,x_{n}) = \prod_{i=1}^{n} (1 - x_{i}).
\]

\subsection{Basis phenomenon: Ring theorem}
The inhomogeneous $q$-Whittaker polynomials satisfy the basic stability property: $$\ff_{\lambda}(x_1,\ldots, x_n, x_{n+1})|_{x_{n+1} = 0} = \ff_{\lambda}(x_1,\ldots, x_n),$$ and hence by taking inverse limit we can consider them in infinitely many variables.  
Note that in infinitely many variables the functions $\mathbb{F}_{\lambda} = \mathbb{F}_{\lambda}(x_1,x_2,\ldots)$ become elements of the completion $\hat\Lambda$ of the ring of symmetric functions $\Lambda$ which contain infinite linear combinations 
of its basis elements with growing degrees. 
More concretely, we have $$\mathbb{F}_{\lambda} = W_{\lambda} + \text{higher degree elements},$$ and such series expansions in the `basis' $\{W_{\lambda} \}$ of $q$-Whittaker functions are infinite. Note that $q$-Whittaker functions $\{W_{\lambda} \}$ form a (homogeneous) basis of $\Lambda$ (over $\mathbb{R}(q)$). In particular, the elements $\{\mathbb{F}_{\lambda} \}$ are linearly independent and the elements of $\hat\Lambda$ can also be expressed as infinite linear combinations of $\{\mathbb{F}_{\lambda} \}$. Given these properties of inhomogeneous $q$-Whittaker functions, the following result is rather remarkable, as a priori any products of these functions should expand infinitely in the same family. 

\begin{thm}[Theorem~\ref{thm:finite_LR_rule}]
(i) There is a commutative ring $\Gamma^q$ (subring of $\hat\Lambda$) with a basis $\{\mathbb{F}_{\lambda}\}$: 
$$
\Gamma^q = \bigoplus_{\lambda} \mathbb{R}(q) \cdot \mathbb{F}_{\lambda},
$$
with finite product expansions given by 
$$
\mathbb{F}_{\mu} \cdot \mathbb{F}_{\nu} = \sum_{\lambda} c^{\lambda}_{\mu, \nu} \,\,\mathbb{F}_{\lambda}, \quad c^{\lambda}_{\mu, \nu} \in \mathbb{R}(q),
$$
for all $\mu, \nu$. 

(ii) Moreover, $\Gamma^q$ can be made a commutative and cocommutative bialgebra with the above product and the coproduct $\Delta: \Gamma^q \to \Gamma^q \otimes \Gamma^q$ given by 
$$
\Delta(\ff_{\lambda}) = \sum_{\mu} \ff_{\lambda/\mu} \otimes \ff_{\mu}
$$
for skew functions $\mathbb{F}_{\lambda/\mu} \in \Gamma^q$ which are also elements of the ring $\Gamma^q$. 
\end{thm}

This result is an instance of the {\it basis phenomenon}, which so far was known to hold only for stable Grothendieck functions and their shifted versions. For stable Grothendieck functions such ring properties were a byproduct of the Littlewood--Richardson (LR) rule proved by Buch \cite{Gcom:B2002}. In \cite{Gcom:Y2019SGrothCauchy} the second author gave another proof of this ring theorem without relying on LR rules. Similarly, for shifted versions of stable Grothendieck functions such ring properties were first conjectured by Ikeda and Naruse \cite{Gcom:IkedaNaruse2013}, and then established for one such family by Clifford--Thomas--Yong in \cite{Gcom:CTYong2014} using LR rule, and proved for another shifted family by Lewis--Marberg \cite{Gcom:LewisMarberg2024} with a proof idea similar to \cite{Gcom:Y2019SGrothCauchy}. 

Let us note that (for $q = 0$) our theorem generalizes the  ring theorem for stable Grothendieck functions. Our proof relies on applying various Cauchy identities in a novel way along with  properties of several dual functions.  The principal difference and difficulty  compared to the proof of ring theorem for stable Grothendieck polynomials from \cite{Gcom:Y2019SGrothCauchy} is that  for inhomogeneous $q$-Whittaker polynomials there is no conjugation automorphism which was essential for stable Grothendieck polynomials.

\subsection{Positive specializations}
As we have established existence of the ring $\Gamma^q$, we can study its positivity. We study positive specializations of the ring $\Gamma^q$ assuming $|q| < 1$ is a fixed real number throughout. A ring homomorphism $\varphi : \Gamma^q \to \mathbb{R}$ is called an {\it $\mathbb{F}$-positive specialization} if $\varphi(\mathbb{F}_{\lambda/\mu}) \ge 0$ for all $\lambda, \mu$. 

Studying this problem for the ring $\Gamma^q$ carries some additional technical obstacles compared to similar problems for the ring $\Lambda$ of symmetric functions. 
For instance, 
we do {\it not } know if $\Gamma^q$ is a polynomial ring with free generators, 
unlike in the case of the polynomial ring 
$\Lambda$ for which positive specializations are described via values or generating functions on its generators such as power sum or elementary symmetric functions, as manifested in the known example of the celebrated Edrei--Thoma theorem on description of Schur-positive specializations and its generalizations for some other bases of $\Lambda$, see e.g. \cite{Borodin_Olshanski_2016, positive:M2019Kerov} and references therein.  Here for the ring $\Gamma^q$ we generally rely on a different approach for describing positive specializations via extensions and unions of some basic generating positive specializations, as was developed in \cite{positive:Yel20}; see also \cite{marberg2025positive} for a recent work addressing this approach. 

    For specializations $\varphi_1, \varphi_2 : \Gamma^{q} \to \mathbb{R}$, we can define the {\it union} of $\varphi_1, \varphi_2$ as the specialization $\varphi$ denoted by $\varphi = (\varphi_1, \varphi_2) = \varphi_1 \cup \varphi_2$ 
    via the bialgebra coproduct $\Delta: \Gamma^q \to \Gamma^q \otimes \Gamma^q$ as 
    $$
    \varphi = (\varphi_1, \varphi_2) = m \circ (\varphi_1 \otimes \varphi_2) \circ \Delta,  
    $$
    where $m : \mathbb{R} \otimes \mathbb{R} \to \mathbb{R}$ is the multiplication map.  
    It is given on the basis elements by the branching formula 
    $$
    \varphi(\ff_{\lambda}) = \ff_{\lambda}(\varphi_1, \varphi_2) = \sum_{\mu} \ff_{\lambda/\mu}(\varphi_1) \ff_{\mu}(\varphi_2), 
    $$
which is compatible with bialgebra structure of $\Gamma^q$. 

\medskip

We show that $\ff$-positive specializations can be described via three basic {\it $\ff$-positive generators}: 
\begin{itemize}
\item Simple variable substitution $\hat\theta_{\boldsymbol{\alpha}}(\ff_{\lambda/\mu}) = \ff_{\lambda/\mu}(\boldsymbol{\alpha})$.
\item Dual variable substitution $\hat\theta'_{\boldsymbol{\beta}} = \hat\theta_{\boldsymbol{\beta}} \circ \omega_q$ given by composing simple variable substitution with $\omega_q$ automorphism (see Prop~\ref{prop:omega_inhomo} below).
\item Plancherel specialization given by $\hat\theta_{Pl, \gamma} = \lim_{n \to \infty}(\underbrace{\hat\theta_{\gamma/n}, \ldots, \hat\theta_{\gamma/n}}_{n \text{ times}})$.
\end{itemize}
These specializations present extensions to similar known Schur-positive generators. 

\medskip

It will be useful to parametrize $\ff$-positive specializations via sequences from the following infinite simplex:  
\begin{multline*}
\Omega := \left\{(\boldsymbol{\alpha}, \boldsymbol{\beta}, \gamma) \in [0,1]^{\infty} \times \mathbb{R}^{\infty}_{\ge 0}  \times \mathbb{R}_{\ge 0} : \boldsymbol{\alpha} =(\alpha_1 \ge \alpha_2 \ge \cdots),\right. \\ \left.
\boldsymbol{\beta} =(\beta_1 \ge \beta_2 \ge \cdots), \sum_{n} (\alpha_n + \beta_n) < \infty \right\}.
\end{multline*}

For a ring homomorphism $\varphi : \Gamma^q \to \mathbb{R}$ we also denote $\varphi_n := \varphi(\ff_{1^n})$ for its values on the single column elements, for which we also have $1 \ge \varphi_1 \ge \varphi_2 \ge \cdots \ge 0$ if $\varphi$ is $\ff$-positive.  

\medskip

We then describe $\ff$-positive specializations as follows. 

\begin{thm}[Theorem~\ref{thm:fact}]
We have: $\varphi : \Gamma^q \to \mathbb{R}$ is $\mathbb{F}$-positive specialization such that     
$
\sum_{n} \varphi_n < \infty,
$
if and only if $\varphi$ can be factorized via $\ff$-positive generators as  
$$
\varphi = \hat\theta_{\boldsymbol{\alpha}} \cup \hat\theta'_{\boldsymbol{\beta}} \cup \hat\theta_{Pl, \gamma} 
$$
for parameters $(\boldsymbol{\alpha}, \boldsymbol{\beta}, \gamma) \in \Omega$ 
such that $\sum_n \varphi_n = \sum_{n} \alpha_n + (1-q)\sum_n\beta_n + \gamma < \infty$. 
\end{thm}

This is a $q$-extension of a result from \cite{positive:Yel20} about Grothendieck-positive specializations. 

As we will show, $\ff$-positive specializations are related to Whittaker-positive specializations of the ring $\Lambda$ of symmetric functions. 
The latter are a special case of Macdonald-positive specializations which were originally conjectured by Kerov and proved by Matveev in \cite{positive:M2019Kerov}.


\subsection{Related probability distributions}

We also show that $\ff$-positive specializations of the ring $\Gamma^q$ produce some interesting probability distributions on integer partitions. 

\begin{thm}[Theorem~\ref{thm:prob}]
Let $q \in (0,1)$ and $\varphi : \Gamma^q \to \mathbb{R}$ be $\mathbb{F}$-positive specialization such that $\varphi_1 < 1$ and $\sum_n \varphi_n < \infty$. Then for any fixed partition $\mu$ there is a well-defined probability distribution on the set of partitions  $\lambda \supseteq \mu$ given by 
$$\mathbb{P}_{\varphi, \mu}(\lambda) := \frac{c_{\lambda, \mu}\, \varphi(\mathbb{F}_{\lambda/\mu})}{Z_{\varphi}}, \qquad c_{\lambda, \mu} := \frac{\prod_{i=1}^{\ell(\lambda')} {(q;q)_{m_{i}(\lambda')}}}{\prod_{i=1}^{\ell(\mu')} {(q;q)_{m_{i}(\mu')}}}, 
$$ 
and 
normalization constant $Z_{\varphi}$ depending on $\varphi$. 
\end{thm}

The proof of this result is based on the factorization of $\ff$-positive specializations shown above, and certain Littlewood-type identity for inhomogeneous $q$-Whittaker polynomials. 

\section{Preliminaries}
\subsection{Notation}
 A \emph{partition} is an non-increasing sequence $\lambda=(\lambda_{1}\geq \lambda_{2}\geq \cdots)$ of positive integers with finitely many non-zero terms. For a partition $\lambda$, the number of non-zero terms  is referred as \emph{length of the partition} which is denoted by \(\ell(\lambda)\), and the sum of all the terms of $\lambda$ is denoted by $|\lambda|$. Furthermore, we define $m_i(\lambda)$ to be the number of parts of $\lambda$ of size $i$.
\

We associate a partition with its corresponding \emph{Young diagram} which is obtained by placing $\lambda_{i}$ boxes in the $i^{th}$ row. The column indices of the Young diagrams increases from left to right while the row indices increases from top to bottom. We denote the \emph{conjugate} of $\lambda$ by $\lambda'$ which is obtained by reflecting the Young diagram of $\lambda$ along the main diagonal.
\

We say $\mu \subseteq \lambda$ whenever the Young diagram of $\mu$ is contained in the Young diagram of $\lambda$. When $\mu \subseteq \lambda$, we define \emph{skew partition} $\lambda/\mu$ as a set theoretic difference of the corresponding  Young diagrams. We write $|\lambda/\mu|$ to denote $|\lambda|-|\mu|$. We write $\mu\prec \lambda$ whenever $\lambda/\mu$ is a horizontal strip i.e., at most one box in each column.
\

Finally, for a positive integer $k$ and formal parameters $x,q$, we recall the $q$-Pochhammer symbol:
\[
(x;q)_k := \prod_{i=0}^{k-1} (1 - x q^i).
\]
\subsection{Macdonald polynomials}
We denote by $\Lambda$ the algebra of symmetric power series of bounded degree in countably many variables $\mathbf{x} =(x_{1}, x_{2}, \dots)$ over $\mathbb{Q}(q, t)$. The \emph{Macdonald symmetric functions} $P_{\lambda}(x_{1}, x_{2}, \dots; q, t)$ form a basis of $\Lambda$. We do not provide a detailed background on these polynomials here (we refer the reader to~\cite{Macdonald}), but recall only the essential aspects required to discuss their positive specializations from~\cite{positive:M2019Kerov}. 

For a partition $\lambda$, 
let:
\begin{equation*}
b_{\lambda}(q,t):=\prod_{(i,j)\in\lambda}\mathsf{b}_{\lambda}(i,j),  
\,\text{ where }\,
\mathsf{b}_{\lambda}(i,j):=\,\dfrac{1-q^{\lambda_{i}-j}t^{\lambda^{'}_{j}-i+1}}{1-q^{\lambda_{i}-j+1}t^{\lambda^{'}_{j}-i}}.   
\end{equation*}
 (If a box $(i,j)$ is outside $\lambda$ we also set $\mathsf{b}_{\lambda}(i,j)=1$.) 
 
 For $\mu\subset\lambda$, we denote by $R_{\lambda/\mu}$ (respectively by $C_{\lambda/\mu}$) the union of all rows (respectively columns) containing boxes from $\lambda/\mu$.
For partitions $\mu \prec \lambda$, let:
\begin{equation*}
\psi_{\lambda/\mu}(q,t):=\prod_{(i,j)\in R_{\lambda/\mu}-C_{\lambda/\mu}}\dfrac{\mathsf{b}_{\mu}(i,j)}{\mathsf{b}_{\lambda}(i,j)}, \qquad 
\overline{\psi}_{\lambda/\mu}(q,t):=\prod_{(i,j)\in  C_{\lambda/\mu}}\dfrac{\mathsf{b}_{\lambda}(i,j)}{\mathsf{b}_{\mu}(i,j)}.
\end{equation*}

 \begin{defn}
 For two partitions $\lambda$ and $\mu$, the single variable skew Macdonald polynomial is defined as:
 \begin{equation}
     P_{\lambda/\mu}(x;q,t)\,=\,\mathbf{1}_{\lambda\prec \mu}\,\psi_{\lambda/\mu}(q,t)\,x^{|\lambda/\mu|}
 \end{equation}
 The $n$-variable \emph{skew Macdonald polynomial} are then defined inductively by
\begin{equation}\label{eq:branching_macd}
P_{\lambda/\mu}(x_{1},\dots,x_{n};q,t)
:= \sum_{\mu \subseteq \nu \prec \lambda}
P_{\nu/\mu}(x_{1},\dots,x_{n-1};q, t)\,
P_{\lambda/\nu}(x_{n};q, t).
\end{equation}
 \end{defn}



\medskip

The \emph{dual Macdonald polynomial} is defined as
\begin{equation}\label{eq:P_Q_relation}
Q_{\lambda/\mu} \;:=\; \frac{b_{\lambda}(q,t)}{b_{\mu}(q,t)}\, P_{\lambda/\mu}.
\end{equation}
We now collect some necessary properties of the polynomials $P_{\lambda}$. They form a basis for the ring of symmetric functions $\Lambda$. At $q=0$, Macdonald functions reduce to Hall--Littlewood polynomials and at $t=0$ they reduce to $q$-Whittaker polynomials, which are discussed below. When $q=t$, both families $P_{\lambda/\mu}$ and $Q_{\lambda/\mu}$ specialize to the skew Schur polynomials $s_{\lambda/\mu}$.

\medskip
We recall that $\Lambda$ denotes the ring of symmetric functions with coefficients in $\mathbb{Q}(q,t)$, which is a polynomial ring with free generators given by power sum symmetric functions $\{p_n \}$, or elementary symmetric functions $\{e_n \}$, i.e. 
$$
\Lambda \cong \mathbb{Q}(q,t)[p_1, p_2, \ldots] \cong \mathbb{Q}(q,t)[e_1, e_2, \ldots]. 
$$
There is an involutive automorphism $\omega_{q,t}:\Lambda_{} \to \Lambda$, which can be defined on the power-sum symmetric functions $p_{n}$ by
\[
\omega_{q,t} : p_{n} \longmapsto (-1)^{n-1} \,\dfrac{1-t^{n}}{1-q^{n}}\,p_{n}.
\] 

For any partitions $\mu \subseteq \lambda$, the involution acts on skew Macdonald polynomials as follows:
\medskip
\begin{equation}\label{eq:omega_qt}
\omega_{q,t}\bigl(P_{\lambda/\mu}(\mathbf{x};q,t)\bigr)
   = Q_{\lambda'/\mu'}(\mathbf{x};t,q),
   \qquad
\omega_{q,t}\bigl(Q_{\lambda/\mu}(\mathbf{x};q,t)\bigr)
   = P_{\lambda'/\mu'}(\mathbf{x};t,q).
\end{equation}

\medskip
This involution maps relates a Cauchy identity to the dual Cauchy identity:

\begin{equation}\label{eq:Macd_Cauchy}
\sum_{\lambda}\,P_{\lambda}(x_{1},x_{2},\dots;q,t)\,Q_{\lambda}(y_{1},y_{2},\dots;q,t)\,=\,\prod_{i,j\geq 1}\,\dfrac{(tx_{i}y_{j};q)_{\infty}}{(x_{i}y_{j};q)_{\infty}}
\end{equation}
and then applying the involution map to the Macdonald polynomial corresponding to the $y$ variables to obtain the following:
\begin{equation*}
\sum_{\lambda}P_{\lambda}(x_{1},x_{2},\dots;q,t)\,P_{\lambda'}(y_{1},y_{2},\dots;t,q)\,=\,\prod_{i,j\geq 1}\,(1+x_{i}y_{j}).
\end{equation*}

\subsection{Macdonald-positive specializations}

\begin{defn}
A {\it specialization} of $\Lambda$ is any homomorphism $\Lambda \to \mathbb{R}$. A specialization $\phi :\Lambda\mapsto\mathbb{R}$ is said to be $(q,t)$-\emph{Macdonald positive specialization} when $\phi(P_{\lambda})\geq 0$ for all $\lambda$.  
\end{defn}

 As $Q_{\lambda}=b_{\lambda}P_{\lambda}$ we know that every positive specialization of $P_{\lambda}$ is a positive specialization of $Q_{\lambda}$. If $\phi$ is $(q,t)$-Macdonald positive specialization, then the {\it dual} specialization $\omega_{t,q}(\phi):=\phi\circ \omega_{t,q}$ is a $(t,q)$-Macdonald positive specialization due to~\eqref{eq:omega_qt}. It immediately follows that $\omega_{q,t}(\omega_{t,q}(\phi))=\phi$.


\begin{defn}
 Given two 
 specializations 
 $\phi_{1}$ and $\phi_{2}$ of $\Lambda$, we define their union $\phi=(\phi_{1},\phi_{2})=\phi_{1}\cup\phi_{2}$ specialization of $\Lambda$ via power sum symmetric functions as 
 $$
\phi(p_n) = \phi_1(p_n) + \phi_2(p_n), \quad n = 1,2,\ldots
 $$
 so that for all $\lambda, \nu$ we have
\begin{equation*}
\phi(P_{\lambda/\nu})\,=\,\sum_{\nu\subseteq\mu\subseteq\lambda}\phi_{1}(P_{\lambda/\mu})\, \phi_{2}(P_{\mu/\nu}) ,\text{\hspace{7mm}and\hspace{7mm}}\phi(Q_{\lambda/\nu})\,=\,\sum_{\nu\subseteq\mu\subseteq\lambda}\phi_{1}(Q_{\lambda/\mu})\,\phi_{2}(Q_{\mu/\nu}).
\end{equation*}
\end{defn}

We observe that:
\begin{equation*}
\omega_{q,t}(\phi_{1},\phi_{2})\,=\,(\omega_{q,t}(\phi_{1}),\omega_{q,t}(\phi_{2})).    
\end{equation*}

\medskip

The following are the basic generating $(q,t)$-Macdonald positive specializations with $\phi(P_{\lambda/\mu})\geq 0$ for any partitions $\mu\subset \lambda$:
\begin{itemize}
    \item {\it Simple variable substitution} $\theta_{\alpha} \text{ for any } \alpha \geq 0 \text{ defined by setting } 
x_{1}\mapsto \alpha,\; x_{i}\mapsto 0 \text{ for } i\geq 2$. More precisely, we have: 
\[
\theta_{\alpha}\bigl(P_{\lambda/\mu}\bigr)
=
\begin{cases}
\psi_{\lambda/\mu}(q,t)\,\alpha^{|\lambda|-|\mu|}, & \text{if } \mu \prec \lambda,\\[4pt]
0, & \text{otherwise.}
\end{cases}
\]

\item {\it Dual variable substitution} $\theta'_{\beta} := \theta_{\beta} \circ \omega_{q,t} = \omega_{q,t}(\theta_{\beta})$ for any $\beta\geq 0$.  Then we have:
\[
\theta'_{\beta}\bigl(P_{\lambda/\mu}\bigr)
=
\begin{cases}
\overline{\psi}_{\lambda'/\mu'}(t,q)\,\beta^{|\lambda|-|\mu|}, & \text{if } \mu' \prec \lambda',\\[4pt]
0, & \text{otherwise.}
\end{cases}
\]

\item {\it Plancherel specialization} given by  $\theta_{Pl,\gamma}:=\lim_{m\mapsto\infty} (\underbrace{\theta_{\gamma/m}, \ldots, \theta_{\gamma/m}}_{m \text{ times}})$ for any $\gamma\geq 0$. We have: 
$$
\theta_{Pl,\gamma}(P_{\lambda/\mu}) = \lim_{m \to \infty} P_{\lambda/\mu}(\underbrace{\gamma/m, \ldots, \gamma/m}_{m \text{ times}}) = \dfrac{\gamma^{|\lambda/\mu|}}{|\lambda/\mu|!}\,\sum_{T\in \operatorname{SYT}(\lambda/\mu)}\psi_{T}(q,t),
$$
where $\operatorname{SYT}(\lambda/\mu)$ denotes the set of {standard} Young tableaux of shape $\lambda/\mu$, that is, semistandard tableaux in which each entry $1,\dots,|\lambda/\mu|$ appears exactly once. (The latter formula is a standard well-known property of specializations of symmetric functions.)
\end{itemize}

\medskip
We now state the result which classifies all the Macdonald-positive specializations (this description is equivalent to the one from \cite{positive:M2019Kerov}):

\begin{thm}[\cite{positive:M2019Kerov}]
\label{thm:Macdpositive}
Let $q, t \in \mathbb{R}$ with $|q|, |t| < 1$. A specialization $\phi : \Lambda \to \mathbb{R}$ is a $(q,t)$-Macdonald positive  specialization if and only if $\phi$ can be factorized as 
\[
\phi\,=\,\theta_{Pl,\gamma}\cup\theta_{\alpha_{1}}\cup\theta_{\alpha_{2}}\cup\cdots \cup\theta'_{\beta_{1}} \cup  \theta'_{\beta_{2}} \cup\cdots
\]
for parameters $\alpha_1 \ge \alpha_2 \ge \cdots \ge 0$, $\beta_1 \ge \beta_2 \ge \cdots \ge 0$, $\gamma \ge 0$ such that 
$
\sum_{n} \bigl(\alpha_{n} + \beta_{n}\bigr) < \infty.
$
\end{thm}

\subsection{Hall--Littlewood and $q$-Whittaker polynomials}
In this subsection, we discuss two specializations of Macdonald polynomials: at \(q=0\) they reduce to \emph{Hall--Littlewood polynomials}, and at \(t=0\) they reduce to \emph{$q$-Whittaker polynomials}.  
We define Hall--Littlewood and $q$-Whittaker polynomials using their branching formulas.

\begin{defn}
Let $\lambda$ and $\mu$ be partitions.  
The one-variable \emph{skew Hall--Littlewood polynomial} is defined by
\begin{equation*}
Q_{\lambda/\mu}(x;t)
   :=\,\mathbf{1}_{\mu\prec\lambda}\, \kappa_{\lambda/\mu}\,
      x^{|\lambda/\mu|}\,,
\end{equation*}
where
\[
\kappa_{\lambda/\mu}
   \,= \,\prod_{i: m_{i}(\lambda)=m_{i}(\mu)+1}
      \bigl(1 - t^{m_{i}(\lambda)}\bigr).
\]      
The $n$-variable \emph{skew Hall--Littlewood polynomials} are then defined inductively by
\begin{equation*}
Q_{\lambda/\mu}(x_{1},\dots,x_{n};t)
    := \,\sum_{\mu \subseteq \nu \prec \lambda}
       Q_{\nu/\mu}(x_{1},\dots,x_{n-1};t)\,
       Q_{\lambda/\nu}(x_{n};t).
\end{equation*}
\end{defn}

\begin{defn}\label{defn:q_Whit}
Let $\lambda$ and $\mu$ be partitions.  
The one-variable \emph{skew $q$-Whittaker polynomial} is defined by
\begin{equation*}
 W_{\lambda/\mu}(x;q) := \mathbf{1}_{\mu \prec \lambda}\,\,
 x^{|\lambda|-|\mu|}\,
 \frac{(q;q)_{\lambda_{r}-\lambda_{r+1}}}{(q;q)_{\lambda_{r}-\mu_{r}}\,(q;q)_{\mu_{r}-\lambda_{r+1}}}.
\end{equation*}
The $n$-variable \emph{skew $q$-Whittaker polynomials} are then defined inductively by
\begin{equation}\label{eq:branching_whit}
W_{\lambda/\mu}(x_{1},\dots,x_{n};q)
:= \sum_{\mu \subseteq \nu \prec \lambda}
W_{\nu/\mu}(x_{1},\dots,x_{n-1};q)\,
W_{\lambda/\nu}(x_{n};q).
\end{equation}
\end{defn}

These two families are related by specializations of the involution map $\omega_{q,t}$.  
In particular,
\begin{equation}\label{eq:omega_q}
\omega_{q,0}\bigl(W_{\lambda/\mu}(\mathbf{x};q)\bigr)
   = Q_{\lambda'/\mu'}(\mathbf{x};q),
 \qquad
\omega_{0,t}\bigl(Q_{\lambda/\mu}(\mathbf{x};t)\bigr)
   = W_{\lambda'/\mu'}(\mathbf{x};t).
\end{equation}
Throughout the text we denote $W_{\lambda/\mu} = W_{\lambda/\mu}(\mathbf{x}; q)$. 

\medskip

We shall use $q$-Whittaker positive specializations. We say that ring homomorphism $\phi : \Lambda \to \mathbb{R}$ is a {\it $q$-Whittaker-positive} specialization if $\phi(W_{\lambda}) \ge 0$ for all $\lambda$. 
The $q$-Whittaker polynomials arise as $t=0$ specialization of Macdonald polynomials and hence their positive specializations can be obtained from Theorem~\ref{thm:Macdpositive} by setting $t = 0$. We also formulate the following corollary of that result, 
presented in an equivalent form.

\begin{cor}[\cite{positive:M2019Kerov}]\label{cor:wgen}
Let $q \in \mathbb{R}$ with $|q| < 1$. 
A ring homomorphism $\phi : \Lambda \to \mathbb{R}$ is a $q$-Whittaker-positive specialization if and only if $\phi$ can be defined on the ring generators $\{W_{1^n} = e_n\}$ via the following generating series: 
$$
1+\sum_{n = 1}^{\infty} \phi(W_{1^n}) z^n = e^{\gamma z} \prod^{\infty}_{i=1}(1+\beta_{i}z)\,\prod^{\infty}_{j=1}\dfrac{1}{(\alpha_{j}z;q)_{\infty}}
$$
parametrized by $\alpha_1 \ge \alpha_2 \ge \cdots \ge 0$, $\beta_1 \ge \beta_2 \ge \cdots \ge 0$, $\gamma \ge 0$ such that $\sum_{n} (\alpha_n + \beta_n) < \infty$. 
\end{cor}

\subsection{Inhomogeneous $\texorpdfstring{q}{q}$-Whittaker polynomials}
In~\cite{spin-BK2024}, Borodin and Korotkikh introduced \emph{inhomogeneous spin $q$-Whittaker} polynomials $\mathbb{F}^{(\mathbf{a},\mathbf{b})}_{\lambda}$, where $\mathbf{a}=(a_{0},a_{1},\dots)$ and $\mathbf{b}=(b_{0},b_{1},\dots)$. These polynomials are a generalization of $q$-Whittaker polynomials and contain two additional sets of parameters. In the same paper, they explored many of their properties, including the branching formula and Cauchy identities, among many other properties.

In this paper, we are interested in two particular specializations of $\mathbb{F}^{(\mathbf{a},\mathbf{b})}_{\lambda}$. The first specialization $a_{i}=0$ and $b_{i}=1$ is referred to as \textbf{\emph{inhomogeneous $q$-Whittaker polynomials}}~$\mathbb{F}_{\lambda}$, and the other specialization $a_{i}=1$ and $b_{i}=0$ is referred to as \textbf{\emph{interpolation $q$-Whittaker polynomials}} $\widetilde{\mathbb{F}}_{\lambda}$. In what follows, we define these polynomials with their branching formula and recall a Cauchy identity involving these polynomials.

\begin{defn}
Let $\lambda$ and $\mu$ be partitions.  
The one-variable \emph{skew inhomogeneous $q$-Whittaker polynomial} is defined by
\begin{equation*}
 \mathbb{F}_{\lambda/\mu}(x) := \mathbf{1}_{\mu \prec \lambda}\,\,
 x^{|\lambda|-|\mu|}\,
 \prod_{r=1}^{\ell(\lambda)}
 (x;q)_{\mu_{r}-\lambda_{r+1}}\,
 \frac{(q;q)_{\lambda_{r}-\lambda_{r+1}}}{(q;q)_{\lambda_{r}-\mu_{r}}\,(q;q)_{\mu_{r}-\lambda_{r+1}}}.
\end{equation*}
The $n$-variable \emph{skew inhomogeneous $q$-Whittaker polynomials} are then defined inductively by
\begin{equation}\label{eq:branching_inhomo}
\mathbb{F}_{\lambda/\mu}(x_{1},\dots,x_{n})
:= \sum_{\mu \subseteq \nu \prec \lambda}
\mathbb{F}_{\nu/\mu}(x_{1},\dots,x_{n-1})\,
\mathbb{F}_{\lambda/\nu}(x_{n}).
\end{equation}
\end{defn}

The skew polynomials $\mathbb{F}_{\lambda/\mu}$ are symmetric and inhomogeneous.  
Their lowest-degree component coincides with the corresponding skew $q$-Whittaker polynomial.  
The non-skew polynomials $\mathbb{F}_{\lambda}$ also form a basis for the ring of symmetric polynomials in finitely many variables. An additional feature of these polynomials is that their highest degree increases with the number of variables, as illustrated by the following example:

\begin{ex}
We illustrate two basic examples of inhomogeneous $q$-Whittaker polynomials.  
First, for a positive integer $n$, the inhomogeneous $q$-Whittaker polynomial corresponding to the partition $\lambda = (1)$ satisfies
\[
\ff_{1/1}(x_1,\ldots, x_n) = 1 - \mathbb{F}_{\square}(x_{1},\dots,x_{n}) = \prod_{i=1}^{n} (1 - x_{i}).
\]
As another example, the inhomogeneous $q$-Whittaker polynomial in two variables corresponding to the partition $(2,0)$ is given by
\[
\mathbb{F}_{(2,0)}(x_{1},x_{2})
= x_{1}^{2}
   + (1+q)x_{1}x_{2}
   + x_{2}^{2}
   - (1+q)(x_{1}^{2}x_{2} + x_{1}x_{2}^{2})
   + q\,x_{1}^{2}x_{2}^{2}.
\]
\end{ex}

We now turn our attention to \emph{interpolation $q$-Whittaker polynomials}. These polynomials are again symmetric and inhomogeneous, and their highest-degree component coincides with the $q$-Whittaker polynomial. They satisfy certain interpolation properties; although these will not be used here, we refer the reader to~\cite{spin-K2024,spinWhitt-Muc} for further details.

\begin{defn}\label{def:interpolation}
Let $\lambda$ and $\mu$ be partitions.  
The dual one-variable \emph{skew interpolation $q$-Whittaker polynomial} is defined by
\begin{equation*}
\widetilde{\mathbb{F}}_{\lambda/\mu}(x)
:= \mathbf{1}_{\mu \prec \lambda}\,
   x^{|\lambda|-|\mu|}\,
   \prod_{r=1}^{\ell(\lambda)}
   (x^{-1};q)_{\lambda_{r}-\mu_{r}}\,
   \dfrac{(q;q)_{\mu_{r}-\mu_{r+1}}}{(q;q)_{\lambda_{r}-\mu_{r}}\,(q;q)_{\mu_{r}-\lambda_{r+1}}}.
\end{equation*}
The $n$-variable \emph{skew interpolation $q$-Whittaker polynomials} are then defined inductively by
\begin{equation*}
\widetilde{\mathbb{F}}_{\lambda/\mu}(x_{1},\dots,x_{n})
:= \sum_{\mu \subseteq \nu \prec \lambda}
   \widetilde{\mathbb{F}}_{\nu/\mu}(x_{1},\dots,x_{n-1})\,
   \widetilde{\mathbb{F}}_{\lambda/\nu}(x_{n}).
\end{equation*}
\end{defn}

\begin{ex}
We illustrate the interpolation $q$-Whittaker polynomial in two variables corresponding to the partition $(2,0)$:
\[
(q;q)_{2}\,\,\widetilde{\mathbb{F}}_{(2,0)}(x_{1},x_{2})
= (x_{1}-1)(x_{1}-q)
  + (1+q)(x_{1}-1)(x_{2}-1)
  + (x_{2}-1)(x_{2}-q).
\]
\end{ex}

\

We now recall some properties of the inhomogeneous $q$-Whittaker polynomials and the interpolation $q$-Whittaker polynomials. The first proposition gives a vanishing condition for skew polynomials, the second presents a Cauchy-type identity involving the inhomogeneous and interpolation families.

\begin{prop}[{\cite[Cor.~5.3]{K2024}}]\label{cor:vanishing_of_skew_whit}
The polynomials $\ff_{\lambda/\mu}(x_{1},\dots,x_{n})$,  and $\widetilde{\ff}_{\lambda/\mu}(x_{1},\dots,x_{n})$ vanish unless $\mu \subseteq \lambda$ and $\ell(\lambda)-\ell(\mu) \leq n$.   
\end{prop}

\begin{prop}[{\cite[Theorem~B]{K2024}; see also \cite{spinWhitt-Muc}}]\label{prop:cauchy}
The inhomogeneous $q$-Whittaker polynomials and the interpolation $q$-Whittaker polynomials satisfy
\begin{equation}\label{eq:cauchy}
\sum_{\lambda}
  \ff_{\lambda/\mu}(x_{1},\dots,x_{n})\,
  \widetilde{\ff}_{\lambda/\nu}(y_{1},\dots,y_{m})\,=\, \prod_{j=1}^{m}\prod_{i=1}^{n}
    \frac{(x_{i};q)_{\infty}}{(x_{i}y_{j};q)_{\infty}}
   \sum_{\lambda}
    \widetilde{\ff}_{\mu/\lambda}(y_{1},\dots,y_{m})\,
    \ff_{\nu/\lambda}(x_{1},\dots,x_{n}).
\end{equation}
with the left hand sum taken over all partitions that contain both $\mu$ and $\nu$, and on the right hand side sum is taken over all partition that are contained in both $\mu$ and $\nu$.
\end{prop}

\subsection{Inhomogeneous Hall--Littlewood polynomials}
In this subsection, we briefly recall generalizations of Hall--Littlewood polynomials. Borodin first introduced a one-parameter deformation in~\cite{spin-Bor2017}, referred to as \emph{spin Hall--Littlewood polynomials}, and another generalization satisfying interpolation properties was developed in~\cite{factHL}, referred to as \emph{factorial Hall--Littlewood polynomials}. A further extension, introduced in~\cite{Garbali2017Newgeneralisation}, contains both of these families as special cases. Our interest lies in a particular specialization of this latter family. In the recent paper~\cite{GWZJ2025}, the authors discuss Cauchy identities satisfied by these polynomials, and throughout this subsection we follow their presentation.

\begin{defn}
Let $\lambda$ and $\mu$ be partitions.  
The one-variable skew inhomogeneous Hall--Littlewood polynomial is defined by
\begin{equation*}
j_{\lambda/\mu}(x;t)
   := \kappa_{\lambda/\mu}\,
      x^{r(\lambda/\mu)}\,
      \prod_{i\in F_{\lambda/\mu}} \bigl(x + t^{m_{i}(\lambda)}\bigr),
\end{equation*}
where
\begin{align*}
F_{\lambda/\mu}
   &= \bigl\{\, i \;\big|\;
        \text{the $i$-th and $(i+1)$-st columns of $\lambda/\mu$ are nonempty} \,\bigr\}, \\
r(\lambda/\mu)
   &= \text{the number of nonempty rows in the skew diagram }\lambda/\mu.
   \end{align*}
   and
\[
\kappa_{\lambda/\mu}
   \,= \,\prod_{i: m_{i}(\lambda)=m_{i}(\mu)+1}
      \bigl(1 - t^{m_{i}(\lambda)}\bigr).
\]      
The $n$-variable skew inhomogeneous Hall--Littlewood polynomials are then defined inductively by
\begin{equation*}
j_{\lambda/\mu}(x_{1},\dots,x_{n};t)
    := \sum_{\mu \subseteq \nu \prec \lambda}
       j_{\nu/\mu}(x_{1},\dots,x_{n-1};t)\,
       j_{\lambda/\nu}(x_{n};t).
\end{equation*}
\end{defn}

These polynomials are inhomogeneous, and their highest-degree component coincides with the Hall--Littlewood polynomials $Q_{\lambda/\mu}(x_{1},\dots,x_{n})$. They are also referred to as factorial Hall--Littlewood polynomials~\cite{factHL}, and in that setting they additionally satisfy interpolation properties. Since we focus on a particular specialization of this family, we choose to refer to them as \emph{inhomogeneous Hall--Littlewood polynomials} in order to maintain consistency with the terminology used for \emph{inhomogeneous $q$-Whittaker polynomials}.

\begin{ex}
We illustrate the inhomogeneous Hall--Littlewood polynomial in two variables corresponding to the partition $(2,0)$, together with its expansion in the Hall--Littlewood basis $\{Q_\lambda\}$:
\begin{align*}
\,\,j_{(2,0)}(x_{1},x_{2};t)
&=\,(1-t)\,\left(x_{1}(x_{1}+1)+(1-t)x_{1}x_{2}+x_{2}(x_{2}+1)\right)\\
&=\,Q_{2}(x_{1},x_{2};t)\,+\,Q_{1}(x_{1},x_{2};t)
\end{align*}
\end{ex}

\

We now introduce the dual polynomials with which the inhomogeneous Hall--Littlewood polynomials satisfy a Cauchy identity. We first define these polynomials via a branching formula, and then recall the corresponding Cauchy identities together with a proposition which discuss the vanishing of them.

\begin{defn}
Let $\mu \prec \lambda$ be partitions.  
The one-variable skew function $J_{\lambda/\mu}(x)$ is defined by
\begin{equation*}
J_{\lambda/\mu}(x;t)
   :=\,\mathbf{1}_{\mu\prec\lambda}\, \kappa_{\lambda/\mu}\,
      \left(\frac{1}{1+x}\right)^{r(\mu/\widetilde{\lambda})}
      \left(\frac{x}{1+x}\right)^{|\lambda/\mu|}
      \prod_{i \in E_{\lambda/\mu}} \bigl(1 + x\, t^{m_{i}(\lambda)}\bigr),
\end{equation*}
The quantities appearing above are given by
\begin{align*}
E_{\lambda/\mu}
   &= \bigl\{\, i \;\big|\;
        \text{the $i$-th and $(i+1)$-st columns of $\lambda/\mu$ are empty, and }
        m_{i}(\lambda) \neq 0 \,\bigr\}, \\
        \widetilde{\lambda}
   &= (\lambda_{2}, \lambda_{3}, \dots)\ \text{(the partition obtained by removing the first part of $\lambda$)},\\
r(\mu/\widetilde{\lambda})
   &= \text{the number of nonempty rows in the skew diagram }\mu/\widetilde{\lambda}.
\end{align*}

The $n$-variable dual skew inhomogeneous Hall--Littlewood polynomials are then defined inductively by
\begin{equation*}
J_{\lambda/\mu}(x_{1},\dots,x_{n};t)
    := \sum_{\mu \subseteq \nu \prec \lambda}
       J_{\nu/\mu}(x_{1},\dots,x_{n-1};t)\,
       J_{\lambda/\nu}(x_{n};t).
\end{equation*}
\end{defn}

\begin{prop}\label{cor:vanishing_of_skew_HL}
The polynomials $j_{\lambda/\mu}(x_{1},\dots,x_{n};t)$ and $J_{\lambda/\mu}(x_{1},\dots,x_{n};t)$ vanish unless $\mu \subseteq \lambda$ and $\ell(\lambda)-\ell(\mu) \leq n$.   
\end{prop}

\begin{prop}[{\cite[Theorem~3.11]{GWZJ2025}}]\label{prop:HL_cauchy}
The inhomogeneous $q$-Whittaker polynomials and its duals satisfy the following identity:
\begin{multline}\label{eq:HL_cauchy}
\sum_{\lambda}\,\dfrac{b_{\mu}(t)}{b_{\lambda}(t)}\,
  j_{\lambda/\mu}(x_{1},\dots,x_{n};t)\,
  J_{\lambda/\nu}(y_{1},\dots,y_{m};t) \,= \\[0.2em]\,\prod_{j=1}^{m}\prod_{i=1}^{n}
    \dfrac{1-t x_{i}y_{j}}{1-x_{i}y_{j}}\,
   \sum_{\lambda}\,\dfrac{b_{\lambda}(t)}{b_{\nu}(t)}\,
    j_{\nu/\lambda}(x_{1},\dots,x_{n};t)\,  J_{\mu/\lambda}(y_{1},\dots,y_{m};t)\,
\end{multline}
where $b_{\lambda}(t)=\prod^{\ell(\lambda)}_{i=1}(t;t)_{m_{i}(\lambda)}$, and the left hand sum is taken over all partitions that contain both $\mu$ and $\nu$, and on the right hand side sum is taken over all partition that are contained in both $\mu$ and $\nu$.
\end{prop}

\begin{ex}
We compute the single-variable case of the above Cauchy identity. Using the branching formulas, for $k\geq 1$ we obtain:
\begin{align*}
  j_{k}(x;t) = (1-t)\,x(1+x)^{k-1}, \quad
  J_{k}(y;t) = (1-t)\,\left(\frac{y}{1+y}\right)^{k}.
\end{align*}
We then compute:
\begin{multline*}
1+\sum^{\infty}_{i=1}\dfrac{1}{(1-t)}\,\left((1-t)\,x(x+1)^{i-1}\right)\,\left((1-t)\left(\frac{y}{(1+y)}\right)^{i}\right)\,\\
=\,1+(1-t)\frac{xy}{(1+y)}\left(\sum_{i=1}1+\left(\dfrac{(1+x)y}{(1+y)}\right)^{i}\right)\,=\,\dfrac{1-txy}{1-xy}
\end{multline*}
\end{ex}
\

\begin{prop}\label{prop:dual_cauchy}
The inhomogeneous $q$-Whittaker polynomials and the inhomogeneous Hall--Littlewood polynomials satisfy
\begin{multline}\label{eq:dual_cauchy}
\sum_{\lambda}
 \dfrac{b_{\mu}(t)}{b_{\lambda}(t)}\,
  j_{\lambda/\mu}(y_{1},\dots,y_{n};t)\, \ff_{\lambda'/\nu'}(x_{1},\dots,x_{n};t)\,
   \,=\\
   \,\prod_{j=1}^{m}\prod_{i=1}^{n}
    \bigl(1 + x_{i}y_{j}\bigr)
   \sum_{\lambda}
    \dfrac{b_{\lambda}(t)}{b_{\nu}(t)}\,
  j_{\nu/\lambda}(y_{1},\dots,y_{n};t)\,
    \ff_{\mu'/\lambda'}(x_{1},\dots,x_{n};t).
\end{multline}
where $b_{\lambda}(t):=\prod^{\ell(\lambda)}_{i=1}(t;t)_{m_{i}(\lambda)}$.
\end{prop}

\medskip
We end this section with a final proposition that discusses the image of $\ff_{\lambda}$ under $\omega_{q}:=\omega_{q,0}$ map. 
\begin{prop}\label{prop:omega_inhomo}
The map $\omega_{q}$ relates inhomogeneous $q$-Whittaker polynomials and dual inhomogeneous Hall--Littlewood polynomials as follows:
\begin{equation}\label{eq:omega_inhomo}
\omega_{q}(F_{\lambda/\mu})\,=\,J_{\lambda'/\mu'}.    
\end{equation}
\end{prop}
\begin{proof}
First, we note that the branching formulas for 
$j_{\lambda/\mu}$, $J_{\lambda/\mu}$, and $\ff_{\lambda/\mu}$ 
imply the stability property
\[
f_{\lambda/\mu}(x_{1},\dots,x_{n},0)
   = f_{\lambda/\mu}(x_{1},\dots,x_{n}),
   \qquad 
   f\in\{\, j,\,J,\,\ff \,\}.
\]
Therefore they extend naturally from symmetric 
polynomials in finitely many variables to symmetric functions 
in countably many variables.

From the Cauchy identity~\ref{eq:HL_cauchy} and the dual Cauchy identity~\ref{eq:dual_cauchy}, we obtain
\[
\omega_{0,t}\!\left(J_{\lambda/\mu}(\mathbf{x};t)\right)
    = \ff_{\lambda'/\mu'}(\mathbf{x};t).
\]
Then the main statement of the proposition follows by replacing $t$ with $q$ in the above equation and using the fact that 
\[
\omega_{q,0}\circ \omega_{0,q} = \mathrm{Id}.
\]
\end{proof}

\section{Ring Theorem}
Unlike many homogeneous families of symmetric polynomials, the degree of the $\ff_{\lambda}$ grows with the number of variables.  
Remarkably, even though their degree increases with the number of variables, the expansion of their product in terms of themselves remains finite. Our proof of this theorem is a novel application of the two Cauchy identities~\ref{eq:cauchy} and ~\ref{eq:dual_cauchy}.

\begin{thm}\label{thm:finite_LR_rule}
The following statements hold:
\begin{enumerate}
    \item[$(a).$] 
    For all partitions $\mu$ and $\nu$, the product 
    \[
        \mathbb{F}_{\nu} \, \mathbb{F}_{\mu}
        = \sum_{\lambda} c^{\lambda}_{\mu,\nu}\, \mathbb{F}_{\lambda}
    \]
    is a finite expansion.  Equivalently, for fixed $\mu$ and $\nu$, there exist only finitely many partitions $\lambda$ such that $c^{\lambda}_{\mu,\nu} \neq 0$.

    \item[$(b).$] The functions $\mathbb{F}_{\lambda}(x_{1}, x_{2}, \dots)$ form a basis of the commutative ring
    \[
        \Gamma^{q} := \bigoplus_{\lambda} \mathbb{R}(q)\,\mathbb{F}_{\lambda}(x_{1},x_{2},\dots).
    \]

    \item[$(c).$] For any partitions $\mu \subseteq \lambda$, the skew inhomogeneous $q$-Whittaker polynomials $\mathbb{F}_{\lambda/\mu}$ belong to the ring $\Gamma^{q}$.
    Moreover, $\Gamma^q$ can be made a commutative and cocommutative bialgebra with the above product and the coproduct $\Delta: \Gamma^q \to \Gamma^q \otimes \Gamma^q$ given by 
$$
\Delta(\ff_{\lambda}) = \sum_{\mu} \ff_{\lambda/\mu} \otimes \ff_{\mu}
$$
for $\mathbb{F}_{\lambda/\mu} \in \Gamma^q$. 
\end{enumerate}
\end{thm}

\begin{proof}
\medskip
\noindent\textit{Proof of \emph{$(a)$}.} Our goal is to show that the expansion
\[
    \mathbb{F}_{\nu}\,\mathbb{F}_{\mu}
    = \sum_{\lambda} c^{\lambda}_{\mu,\nu}\,\mathbb{F}_{\lambda}
\]
is finite.  
Equivalently, we must prove that only finitely many partitions 
$\lambda = (\lambda_{1},\dots,\lambda_{\ell(\lambda)})$ contribute nonzero terms.  
To establish this, we will show that both the length $\ell(\lambda)$ and the first part $\lambda_{1}$ are bounded, which in turn implies that only finitely many such partitions $\lambda$ can occur.

We begin by proving that $\ell(\lambda)$ is bounded. To prove this, we consider the $\mu=\varnothing$ case of the skew-Cauchy identity~\ref{eq:cauchy}:
\begin{equation*}
\sum_{\lambda}
  \ff_{\lambda}(x_{1},\dots,x_{n})\,
  \widetilde{\ff}_{\lambda/\nu}(y_{1},\dots,y_{m}) \,=\, \prod_{j=1}^{m}\prod_{i=1}^{n}
    \frac{(x_{i};q)_{\infty}}{(x_{i}y_{j};q)_{\infty}}
    \ff_{\nu}(x_{1},\dots,x_{n}).
\end{equation*}

We then expand the Cauchy kernel by taking the $\mu=\nu=\varnothing$ case of the skew-Cauchy identity~\ref{eq:cauchy}:
\begin{align*}
\sum_{\lambda}
  \ff_{\lambda}(x_{1},\dots,x_{n})\,
  \widetilde{\ff}_{\lambda/\nu}(y_{1},\dots,y_{m}) \,&=\,\left(\sum_{\mu}
 \,\ff_{\mu}(x_{1},\dots,x_{n})\,
  \widetilde{\ff}_{\mu}(y_{1},\dots,y_{m})\right)\, \ff_{\nu}(x_{1},\dots,x_{n})\\
  &=\,\sum_{\mu}\left(\sum_{\lambda}\,c^{\lambda}_{\mu,\nu}\,\ff_{\lambda}(x_{1},\dots,x_{n})\right)\widetilde{\ff}_{\mu}(y_{1},\dots,y_{m})\\
  &=\, \sum_{\lambda}\left(\sum_{\mu}\,c^{\lambda}_{\mu,\nu}\,\widetilde{\ff}_{\mu}(y_{1},\dots,y_{m})\right)\ff_{\lambda}(x_{1},\dots,x_{n}).
\end{align*}

By equating the coefficients of $\ff_{\lambda}(x_{1},\dots,x_{n})$ on both sides, we obtain:
\begin{equation*}
  \widetilde{\ff}_{\lambda/\nu}(y_{1},\dots,y_{m})
    = \sum_{\mu} c^{\lambda}_{\mu,\nu}\,
      \widetilde{\ff}_{\mu}(y_{1},\dots,y_{m}).
\end{equation*}

We recall that the partitions $\mu$ and $\nu$ are fixed. Let $\lambda$ be a partition such that 
$c^{\lambda}_{\mu,\nu} \neq 0$. Let $k$ be the minimal integer for which 
$\widetilde{\ff}_{\mu}(y_{1},\dots,y_{k}) \neq 0$. This implies that $\widetilde{\ff}_{\lambda/\nu}\neq 0$. Then, by Corollary~\ref{cor:vanishing_of_skew_whit}, we conclude that  
\[
\ell(\lambda) \leq \ell(\nu) + k.
\]
\smallskip
We now prove that $\lambda_{1}=\ell(\lambda')$ is bounded. To bound $\lambda_{1}$, we repeat the same argument but with the dual Cauchy identity~\ref{prop:dual_cauchy} in the case $\mu=\varnothing$:
\begin{align*}
\sum_{\lambda}\,\dfrac{b_{\nu'}(t)}{b_{\lambda'}(t)}\,
  \ff_{\lambda}(x_{1},\dots,x_{n};t)\,
  &j_{\lambda'/\nu'}(y_{1},\dots,y_{m};t) \\
\quad &=\quad \prod_{j=1}^{m}\prod_{i=1}^{n}
    \bigl(1 + x_{i}y_{j}\bigr)
   \,
    \ff_{\nu}(x_{1},\dots,x_{n})\\[0.5em]
   &=\,\sum_{\mu}\,\dfrac{1}{b_{\mu'}(t)}
  \ff_{\mu}(x_{1},\dots,x_{n})\,
  j_{\mu'}(y_{1},\dots,y_{m})\,
   \,
    \ff_{\nu}(x_{1},\dots,x_{n})\\[0.5em]
   &=\,\sum_{\lambda}\left(\sum_{\mu}\,\dfrac{b_{\lambda'}(t)}{b_{\nu'}(t)\,b_{\mu'}(t)}\,c^{\lambda}_{\mu,\nu}
 \, j_{\mu'}(y_{1},\dots,y_{m})\right)\,
  \mathbb{F}_{\lambda}(x_{1},\dots,x_{m})\,
\end{align*}

We can then conclude that 
\[
j_{\lambda'/\nu'}(y_{1},\dots,y_{m};t)\,=\,\sum_{\mu}\,\dfrac{b_{\lambda'}(t)}{b_{\nu'}(t)\,b_{\mu'}(t)}\,\,c^{\lambda}_{\mu,\nu}\,\,
 \, j_{\mu'}(y_{1},\dots,y_{m};t).
\]dami

For fixed partitions $\mu$ and $\nu$, let $\lambda$ be a partition such that 
$c^{\lambda}_{\mu,\nu} \neq 0$.  
Let $k'$ be the minimal integer for which 
$j_{\mu'}(y_{1},\dots,y_{k'}) \neq 0$. Then it follows that $j_{\lambda'/\nu'}(y_{1},\dots,y_{k'}) \ne 0$. Then, by Proposition~\ref{cor:vanishing_of_skew_HL}, we conclude that  
\[
\lambda_{1}\,=\,\ell(\lambda') \leq \ell(\nu') + k'.
\]


\medskip

\noindent\textit{Proof $(b)$.}
Firstly, we observe that the lowest degree component of $\ff_{\lambda}$ is equal to $q$-Whittaker polynomial:
\[
\ff_{\lambda}\,=\,W_{\lambda}\,+\,\text{ higher order terms.}
\]
This can be seen by comparing the branching formula of $\ff_{\lambda}$ in
Equation~\ref{eq:branching_inhomo} with the corresponding branching formula of
$W_{\lambda}$ in Equation~\ref{eq:branching_whit}.  
In particular, the leading homogeneous components of $\ff_{\lambda}$ coincide with those of
$W_{\lambda}$, which implies that the family $\{\ff_{\lambda}\}$ is linearly independent.  
Moreover, by part~(a) of Theorem~\ref{thm:finite_LR_rule}, the product of any two
$\ff_{\lambda}$ lies in the ring $\Gamma^{q}$.


\medskip

\noindent\textit{Proof of $(c)$.}
We now show that each skew polynomial $\mathbb{F}_{\lambda/\mu}$ can be written as a finite 
$\mathbb{R}(q)$-linear combination of the basis elements $\{\mathbb{F}_{\nu}\}$ and therefore lies in the ring $\Gamma^{q}$.

\smallskip
From the dual Cauchy identity~\ref{eq:dual_cauchy}, we obtain the following correspondence between the coefficients 
\(d^{\lambda}_{\mu,\nu}\) appearing in the expansions:
\begin{equation*}
\ff_{\lambda/\mu}
   = \sum_{\nu}
     \frac{b_{\lambda'}(t)}{b_{\mu'}(t)\,b_{\nu'}(t)}\,
     d^{\lambda'}_{\mu',\nu'}\,\ff_{\nu},
\end{equation*}
and
\begin{equation*}
j_{\mu'}\,j_{\nu'}
   = \sum_{\lambda'} d^{\lambda'}_{\mu',\nu'}\, j_{\lambda'}.
\end{equation*}

Using the same finiteness argument as in the proof of part~(a) of Theorem~\ref{thm:finite_LR_rule}—now applied to the Cauchy identities~\eqref{eq:HL_cauchy} and~\eqref{eq:dual_cauchy}—we see that for any fixed partition $\lambda$, only finitely many pairs $(\mu,\nu)$ satisfy $d^{\lambda}_{\mu,\nu}\neq 0$.  
Hence each $\ff_{\lambda/\mu}$ is a finite linear combination of the $\ff_{\nu}$ and therefore belongs to the ring $\Gamma^{q}$.

The bialgebra property is then a standard consequence of the branching rule for inhomogeneous $q$-Whittaker functions $\ff_{\lambda}$: 
$$
\ff_{\lambda}(\mathbf{x}, \mathbf{y}) = \sum_{\mu} \ff_{\lambda/\mu}(\mathbf{x})\, \ff_{\mu}(\mathbf{y}), \qquad \Delta: \Gamma^q \to \Gamma^q \otimes \Gamma^q, \qquad \Delta : \ff_{\lambda} \mapsto \sum_{\mu} \ff_{\lambda/\mu} \otimes \ff_{\mu}, 
$$
and counit given by the constant term linear map: $\ff_{\varnothing} \mapsto 1$, and $\ff_{\lambda} \mapsto 0$ for $\lambda \ne \varnothing$. 
\end{proof}
\

We give an example illustrating the finiteness of the product of $\ff_{\lambda}$.

\begin{ex}
We illustrate the product of two inhomogeneous $q$-Whittaker polynomials corresponding to the partition $\square=(1)$.  
In two variables, a direct computation gives
\[
\ff_{\square}(x_{1},x_{2})\,\ff_{\square}(x_{1},x_{2})
   \;=\;
   \ff_{(2,0)}(x_{1},x_{2})
   \;+\; (1-q)\,\ff_{(1,1)}(x_{1},x_{2})
   \;-\; (1-q)\,\ff_{(2,1)}(x_{1},x_{2}),
\]
where
\[
\ff_{(2,1)}(x_{1},x_{2})
   = x_{1}^{2}x_{2}
     + x_{1}x_{2}^{2}
     - x_{1}^{2}x_{2}^{2},
\]
and
\[
\mathbb{F}_{(2,0)}(x_{1},x_{2})
= x_{1}^{2}
   + (1+q)x_{1}x_{2}
   + x_{2}^{2}
   - (1+q)(x_{1}^{2}x_{2}+x_{1}x_{2}^{2})
   + q\,x_{1}^{2}x_{2}^{2}.
\]

\medskip
By induction on the number of variables, together with the branching formula, this identity extends to $n$ variables:
\begin{multline*}
\ff_{\square}(x_{1},\dots,x_{n})\,\ff_{\square}(x_{1},\dots,x_{n}) \\
   =\;
   \ff_{(2,0)}(x_{1},\dots,x_{n})
   \;+\; (1-q)\,\ff_{(1,1)}(x_{1},\dots,x_{n})
   \;-\; (1-q)\,\ff_{(2,1)}(x_{1},\dots,x_{n}).
\end{multline*}
\end{ex}

\

We conclude this section with a proposition governing the product
$\ff_{\square}\,\times\,\ff_{\lambda}$:

\begin{prop}\label{prop:Inwhit_Pieri}
Inhomogeneous $q$-Whittaker polynomials satisfy the following product rule:
\begin{equation*}
\ff_{\square}\,\ff_{\nu} \\
   =\; \sum_{\lambda}\,c^{\lambda}_{\square,\nu}\,\ff_{\lambda},
\end{equation*}
where the summation is over partitions $\lambda$ such that $\lambda/\nu$ has at-most one box in every column and every row, and we have
\begin{equation*}
c^{\lambda}_{\square,\nu}:=(-1)^{|\lambda|-|\nu|-1}\,\eta_{\lambda/\nu}\,(1-q)^{|\lambda|-|\nu|}
\end{equation*}
where 
\[
\eta_{\lambda/\nu}\,=\,\prod_{j\geq 1}\,\dfrac{(q;q)_{\nu_{j}-\nu_{j+1}}}{(q;q)_{\lambda_{j}-\nu_{j}}\,(q;q)_{\nu_{j}-\lambda_{j+1}}}.
\]
\end{prop}
\begin{proof}
We begin with the case $\mu=\varnothing$ and $m=1$ of the skew Cauchy identity~\ref{eq:cauchy}:
\begin{equation*}
\sum_{\lambda}
  \ff_{\lambda}(x_{1},\dots,x_{n})\,
  \widetilde{\ff}_{\lambda/\nu}(y)
  \;=\;
  \prod_{i=1}^{n}
    \frac{(x_{i};q)_{\infty}}{(x_{i}y;q)_{\infty}}
  \,\ff_{\nu}(x_{1},\dots,x_{n}).
\end{equation*}

Next, we expand the Cauchy kernel using the specialization $\mu=\nu=\varnothing$ of~\ref{eq:cauchy}:
\begin{align*}
\sum_{\lambda}
  \ff_{\lambda}(x_{1},\dots,x_{n})\,
  \widetilde{\ff}_{\lambda/\nu}(y)
  &=
  \left(
     \sum_{k\ge 0}
       \ff_{k}(x_{1},\dots,x_{n})\,\widetilde{\ff}_{k}(y)
  \right)
  \ff_{\nu}(x_{1},\dots,x_{n}) \\[0.3em]
 &=
  \sum_{k\ge 0}
   \left(
      \sum_{\lambda}
        c^{\lambda}_{k,\nu}\,\ff_{\lambda}(x_{1},\dots,x_{n})
   \right)
   \widetilde{\ff}_{k}(y) \\[0.3em]
 &=
  \sum_{\lambda}
    \left(
      \sum_{k\ge 0}
         c^{\lambda}_{k,\nu}\,\widetilde{\ff}_{k}(y)
    \right)
    \ff_{\lambda}(x_{1},\dots,x_{n}).
\end{align*}

Since $\{\ff_{\lambda}\}$ form a basis, equating the coefficients of 
$\ff_{\lambda}(x_{1},\dots,x_{n})$ on both sides yields:
\begin{equation}\label{eq:Pieri_skew_onevariable}
  \widetilde{\ff}_{\lambda/\nu}(y)
   \;=\;
   \sum_{k\ge 0}
     c^{\lambda}_{k,\nu}\,
     \widetilde{\ff}_{k}(y).
\end{equation}

From Definition~\ref{def:interpolation} we have
\begin{equation}\label{eq:onevariable_formula}
\widetilde{\ff}_{\lambda/\nu}(y)
  = \mathbf{1}_{\nu\prec\lambda}\,\eta_{\lambda/\nu}
    \prod_{r=1}^{\ell(\lambda)}
    \prod_{i=1}^{\lambda_{r}-\nu_{r}} (y-q^{i-1})
 \qquad \text{   and   }\qquad
\widetilde{\ff}_{k}(y)
  = \,\dfrac{1}{(q;q)_{k}}\,\prod_{i=1}^{k}(y-q^{i-1}).
\end{equation}

Setting \(y=1\) in \eqref{eq:Pieri_skew_onevariable}, and using
\(\widetilde{\ff}_{k}(1)=0\) for all \(k>0\), we obtain
\[
c^{\lambda}_{\varnothing,\nu} = \mathbf{1}_{\lambda=\nu}.
\]

\medskip
We now evaluate at \(y=q\).  
Since \(\widetilde{\ff}_{k}(q)=0\) for all \(k>1\), equation~\eqref{eq:Pieri_skew_onevariable}
gives
\begin{equation*}
  \widetilde{\ff}_{\lambda/\nu}(q)
   = c^{\lambda}_{\square,\nu}\,\widetilde{\ff}_{\square}(q).
\end{equation*}
Using
\[
\widetilde{\ff}_{1}(y)
   = \frac{y-1}{1-q}
   \qquad\text{ and }\qquad
\widetilde{\ff}_{\square}(q) = -1,
\]
we obtain
\begin{equation*}
c^{\lambda}_{\square,\nu}
   = -\,\widetilde{\ff}_{\lambda/\nu}(q).
\end{equation*}

From the explicit formula in \eqref{eq:onevariable_formula},  
\(\widetilde{\ff}_{\lambda/\nu}(q)\) vanishes whenever
\(\lambda_{r}-\nu_{r} > 1\) for any \(r\).  
Thus in the skew diagram \(\lambda/\nu\), each row contains at most one box.
Moreover, the condition \(\nu\prec\lambda\) implies that each column
also contains at most one box. 

\medskip
Finally for a skew shape $\lambda/\nu$ such that it contains at-most one box in any given row or column, we compute
\[
\widetilde{\ff}_{\lambda/\nu}(q)\,=\,\,\eta_{\lambda/\nu}\,\prod_{i:\lambda_{i}-\mu_{i}=1}\,(q-1)\,=\,\eta_{\lambda/\nu}\,(-1)^{|\lambda|-|\nu|}(1-q)^{|\lambda|-|\nu|}.
\]
This completes our proof.
\end{proof}

\begin{ex}
We illustrate the above proposition with the following example:
\[
\begin{aligned}
\ff_{\square}\,\ff_{(3,1)}
   \;=\;&\;
      \ff_{(4,1)}
      \;+\;(1 - q^{2})\,\ff_{(3,2)}
      \;+\;(1 - q)\,\ff_{(3,1,1)} \\[0.4em]
   &\;
      -\,(1 - q^{2})\,\ff_{(4,2)}
      \;-\,(1 - q)(1 - q^{2})\,\ff_{(3,2,1)}
      \;-\,(1 - q)\,\ff_{(4,1,1)} \\[0.4em]
   &\;
      +\,(1 - q)(1 - q^{2})\,\ff_{(4,2,1)}.
\end{aligned}
\]

\end{ex}

\section{Positive specializations}
We assume $|q| < 1$ for real $q$ throughout this section. 
    A {\it specialization} of a ring is a homomorphism to $\mathbb{R}$. We say that a ring homomorphism $\varphi : \Gamma^q \to \mathbb{R}$ is {\it $\ff$-positive specialization} if $\varphi(\ff_{\lambda/\mu}) \ge 0$ for all partitions $\lambda$ and $\mu$. We use both notations $\varphi(\ff_{\lambda}) = \ff_{\lambda}(\varphi)$ for specializations. 

\begin{defn}[Union of specializations]
    Let $\varphi_1, \varphi_2 : \Gamma^{q} \to \mathbb{R}$ be specializations of $\Gamma^q$. The {\it union} of $\varphi_1, \varphi_2$ is the specialization $\varphi$ denoted by $\varphi = (\varphi_1, \varphi_2) = \varphi_1 \cup \varphi_2$ 
    defined via the bialgebra coproduct $\Delta: \Gamma^q \to \Gamma^q \otimes \Gamma^q$ as 
    $$
    \varphi = (\varphi_1, \varphi_2) = m \circ (\varphi_1 \otimes \varphi_2) \circ \Delta,  
    $$
    where $m : \mathbb{R} \otimes \mathbb{R} \to \mathbb{R}$ is the multiplication map  $m: a \otimes b \mapsto a\cdot b$. 
\end{defn}

\begin{lem}
Let $\varphi_{1}, \varphi_{2} : \Gamma^q \to \mathbb{R}$ be $\ff$-positive specializations of $\Gamma^{q}$. Then their union $\varphi=(\varphi_{1},\varphi_{2})$ is also $\ff$-positive specialization.
\end{lem}
\begin{proof}
Note that from the branching formula we have $\Delta(\ff_{\lambda/\mu}) = \sum_{\nu} \ff_{\lambda/\nu} \otimes \ff_{\nu/\mu}$ which implies  
    $$
    \varphi(\ff_{\lambda/\mu}) = \ff_{\lambda/\mu}(\varphi_1, \varphi_2) = \sum_{\nu} \ff_{\lambda/\nu}(\varphi_1)\, \ff_{\nu/\mu}(\varphi_2) \ge 0
    $$
and so $\varphi$ is $\ff$-positive. 
\end{proof}

\begin{defn}[Extensions from $\Lambda$ to $\Gamma^q$]
Let $\phi : \Lambda \to \mathbb{R}$ be a specialization of $\Lambda$. 
First, we construct its extension to the ring homomorphism $\phi_z : \Gamma^q \to \mathbb{R}[[z]]$ as follows: 
For an element $F \in \Gamma^q$ which can be expanded as the sum of homogeneous components 
$$
F = \sum_{n} F_{n}, 
\text{ where } F_{n} \in \Lambda \text{ with} \deg(F_n) = n, 
$$
we define the series 
$$
\phi_z : F \longmapsto \sum_{n} \phi(F_n) z^n \in \mathbb{R}[[z]].
$$
We then define the specialization $\hat\phi = S(\phi_z): \Gamma^q \to \mathbb{R}$ as value at $z = 1$ of the unique analytic continuation of convergent series $\phi_z$ to the positive real line. 
In such case, we say that $\phi$ {\it extends} to $\hat\phi$. 
In other words, we consider the specialization $\hat\phi$ of $\Gamma^q$ as an image under the specialization $\phi$ of $\Lambda$. 
\end{defn}

\begin{lem}
Let $\phi_1, \phi_2 : \Lambda \to \mathbb{R}$ be specializations that extend to specializations $\hat\phi_1, \hat\phi_2 : \Gamma^q \to \mathbb{R}$. Then their union $\phi = (\phi_1, \phi_2)$ extends to the union $\hat\phi = (\hat\phi_1, \hat\phi_2)$. 
\end{lem}
\begin{proof}
    Let us first recall that the union of $\phi_1, \phi_2$ can be similarly defined via the standard (degree preserving) coproduct of the ring of symmetric functions $\Delta : \Lambda \to \Lambda \otimes \Lambda$ as $$\phi = (\phi_1, \phi_2) = m \circ (\phi_1 \otimes \phi_2) \circ \Delta,$$
    (which is equivalent to specializing $p_i(\phi)=p_i(\phi_1) + p_i(\phi_2)$ on the power sum symmetric functions $p_i$). 
    This coproduct extends uniquely to continuous homomorphism of completions $\Delta : \hat\Lambda \to \hat\Lambda \otimes \hat\Lambda$ (w.r.t. degree filtration). Hence, it also extend uniquely to the coproduct $\Delta : \Gamma^q \to \Gamma^q \otimes \Gamma^q$ which is a subring of $\hat\Lambda$.  

    For $F \in \Gamma^q$ we write 
    $$
    \Delta(F) = \sum F^{(1)} \otimes F^{(2)}.
    $$
    Recall that we have 
    $$
\phi_z(F) = \sum_n \phi(F_n) z^n, \quad F_n \in \Lambda,\, \deg(F_n) = n.
    $$
Applying the coproduct and the definition of union we obtain that 
\begin{align*}
\phi_z(F) &= \sum \sum_n \sum_{i + j = n} \phi_1(F^{(1)}_i)\, \phi_2(F^{(2)}_j)\, z^{i+j} \\
&= \sum \left(\sum_{i} \phi_1(F^{(1)}_i)\, z^i \right) \cdot \left( \sum_{j} \phi_2(F^{(2)}_j)\, z^j \right) \\
&= \sum \phi_{1, z}(F^{(1)}) \cdot \phi_{2,z}(F^{(2)}).
\end{align*}
Then by analytic continutation $S$ evalutation at $z = 1$ we get 
\begin{align*}
\hat\phi(F) &= S(\phi_z(F)) \\
&= \sum S(\phi_{1,z}(F^{(1)}) \cdot \phi_{2,z}(F^{(2)})) \\
&= \sum S(\phi_{1,z}(F^{(1)})) \cdot S(\phi_{2,z}(F^{(2)})) \\
&= \sum \hat\phi_1(F^{(1)}) \cdot \hat\phi_1(F^{(2)}) \\
&= (\hat\phi_1, \hat\phi_2)(F)
\end{align*}
as needed. 
\end{proof}

Let us recall that we will parametrize $\ff$-positive specializations via sequences from the following infinite simplex:  
\begin{multline*}
\Omega := \left\{(\boldsymbol{\alpha}, \boldsymbol{\beta}, \gamma) \in [0,1]^{\infty} \times \mathbb{R}^{\infty}_{\ge 0}  \times \mathbb{R}_{\ge 0} : \boldsymbol{\alpha} =(\alpha_1 \ge \alpha_2 \ge \cdots),\right. \\ \left.
\boldsymbol{\beta} =(\beta_1 \ge \beta_2 \ge \cdots), \sum_{n} (\alpha_n + \beta_n) < \infty \right\}.
\end{multline*}

For a specialization $\varphi : \Gamma^q \to \mathbb{R}$ we also denote $\varphi_n := \varphi(\ff_{1^n})$ for the values on single column elements. 

Our main description is based on the following transitions between $\ff$-positive and $W$-positive specializations. 

\begin{thm}
\label{thm:F-positive_specializations}
    There is a bijection between 
    the set of $\mathbb{F}$-positive specializations $\varphi : \Gamma^q \to \mathbb{R}$ 
    such that 
    $
    \sum_{n} \varphi_n < \infty,
    $
    and the set of $W$-positive specializations of the ring $\Lambda$ parametrized by the sequences $(\boldsymbol{\alpha}, \boldsymbol{\beta}, \gamma) \in \Omega$. 
    Namely, we have the following: 

(i) Let $\phi : \Lambda \to \mathbb{R}$ be $W$-positive specialization parametrized by $(\boldsymbol{\alpha}, \boldsymbol{\beta}, \gamma) \in \Omega$. Then $\phi$ extends to $\ff$-positive specialization $\varphi = \hat\phi : \Gamma^q \to \mathbb{R}$ such that 
$$\sum_{n} \varphi_n = \sum_n \alpha_n + (1-q)\sum_n\beta_n + \gamma < \infty.$$

(ii) Conversely, let $\varphi : \Gamma^q \to \mathbb{R}$ be $\ff$-positive specialization such that $\sum_n \varphi_n < \infty$. Then the ring homomorphism $\phi : \Lambda \to \mathbb{R}$ given on the ring generators $\{ W_{1^n} \}$ by 
$$
\phi : W_{1^n} \longmapsto \varphi_n + \binom{n}{1} \varphi_{n+1} + \binom{n+1}{2} \varphi_{n+2} + \cdots,  n = 1,2,\ldots,  
$$
defines $W$-positive specialization parametrized by $(\boldsymbol{\alpha}, \boldsymbol{\beta}, \gamma) \in \Omega$. Furthermore, $\phi$ extends back to $\varphi$. 
\end{thm}

Let us recall the basic generating $W$-positive specializations: 
\begin{itemize}
\item Simple variable substitution $\theta_{{\alpha}} : W_{\lambda} \mapsto W_{\lambda}({\alpha})$.
\item Dual variable substitution $\theta'_{{\beta}} = \theta_{{\beta}} \circ \omega_q$ given by composing simple variable substitution with $\omega_q$ automorphism, so that $\theta'_{{\beta}} : W_{\lambda} \mapsto Q_{\lambda'}(\beta; q)$. 
\item The Plancherel specialization given by $\theta_{Pl, \gamma} = \lim_{n \to \infty}(\underbrace{\theta_{\gamma/n}, \ldots, \theta_{\gamma/n}}_{n \text{ times}})$.
\end{itemize}

\begin{lem}\label{lem:genf}
Let $\alpha, \beta, \gamma \ge 0$ and $\alpha \le 1$. Then $W$-positive generators $\theta_{\alpha}, \theta'_{\beta}, \theta_{Pl, \gamma} : \Lambda \to \mathbb{R}$ extend to $\ff$-positive specializations $\hat\theta_{\alpha}, \hat\theta'_{\beta}, \hat\theta_{Pl,\gamma} : \Gamma^q \to \mathbb{R}$. 
\end{lem}
\begin{proof}
We are going to check the corresponding extensions on the basis elements $\{\ff_{\lambda} \}$. 

(i) Let $\alpha \in [0,1]$. Consider the expansion of inhomogeneous $q$-Whittaker polynomials in the basis of $q$-Whittaker polynomials: 
$$\ff_{\lambda} = \sum_{\mu} \tilde{a}_{\lambda, \mu}\, W_{\mu}$$ and its $\theta_{\alpha}$ specialization given by 
$$
\hat\theta_{\alpha} : \ff_{\lambda} \longmapsto \sum_{\mu} \tilde{a}_{\lambda, \mu} W_{\mu}(\alpha), 
$$
which vanishes unless $\lambda = (n)$ is a single-row partition, for which we get that 
$$
\hat\theta_{\alpha} : \ff_{(n)} \longmapsto \sum_{\mu} \tilde{a}_{(n), \mu}\, W_{\mu}(\alpha)
$$
Note that $\tilde{a}_{(n), \mu} W_{\mu}(\alpha) = 0$ unless $\mu = (n)$ for which $\tilde{a}_{(n), (n)} = 1$, and hence $\hat\theta_{\alpha} : \ff_{(n)} \mapsto W_{(n)}(\alpha)$. 
Since the above series converges, then we have 
$$\hat\theta_{\alpha} : \ff_{\lambda} \longmapsto \ff_{\lambda}(\alpha) = \sum_{\mu} \tilde{a}_{\lambda, \mu} W_{\mu}(\alpha),$$ 
i.e. $\hat\theta_{\alpha}$ just corresponds to simple variable substitution $x_1 \mapsto \alpha, x_i \mapsto 0$ for $i \ge 2$. Hence,  
$$
\hat\theta_{\alpha}  (\ff_{\lambda/\mu}) = \ff_{\lambda/\mu}(\alpha)= \alpha^{|\lambda/\mu|}\prod_{i\geq 1}(\alpha;q)_{\mu_{i+1}-\lambda_{i}} \ge 0.
$$

(ii) Let $\beta \ge 0$. Consider now the dual expansion 
$$\omega_q (\ff_{\lambda}) = \sum_{\mu} \tilde{a}_{\lambda, \mu}\, \omega_q(W_{\mu}) = \sum_{\mu} \tilde{a}_{\lambda, \mu}\, Q_{\mu'}$$
and its dual $\theta'_{\beta}$ specialization given by 
$$
\hat\theta'_{\beta} : \ff_{\lambda} \longmapsto \sum_{\mu} \tilde{a}_{\lambda, \mu}\, Q_{\mu'}(\beta),
$$
which vanishes unless $\lambda = (1^n)$ is a single-column partition, for which we get that 
$$
\hat\theta'_{\beta} : \ff_{(1^n)} \longmapsto \sum_{\mu} \tilde{a}_{(1^n), \mu}\, Q_{\mu'}(\beta)
$$
From the following expansion in Prop:~\ref{prop:wf}
$$
\ff_{1^{n}}\,=\,W_{1^{n}}-\binom{n}{1}W_{1^{n+1}}+\binom{n+1}{2}\,W_{1^{n+2}}-\binom{n+2}{3}\,W_{1^{n+3}}+\cdots
$$
we obtain that 
$$
\omega_{q}(\ff_{1^{n}})\,=\,Q_{n}-\binom{n}{1}Q_{n+1}+\binom{n+1}{2}\,Q_{{n+2}}-\binom{n+2}{3}\,Q_{{n+3}}+\cdots
$$
Note that $\theta_{\beta}(Q_n) = Q_n(\beta) = (1-q)\beta^n$ and the series 
$$
(1-q)\left(\beta^n-\binom{n}{1}\beta^{n+1}+\binom{n+1}{2}\,\beta^{n+2}-\binom{n+2}{3}\,\beta^{{n+3}}+\cdots\right)
$$
which is convergent for $\beta < 1$ has unique analytic continuation to the function $(1-q)\left(\frac{\beta}{1 + \beta} \right)^n$ for all $\beta > 0$. 
Hence, we define the extension 
$$
\hat\theta'_{\beta} : \ff_{1^n} \longmapsto \,(1-q)\left(\frac{\beta}{1 + \beta} \right)^n
$$
and then by equation~\eqref{eq:omega_inhomo}, we have
$$
\hat\theta'_{\beta}(\ff_{\lambda/\mu}) = \omega_q(\ff_{\lambda/\mu})(\beta)  =  \kappa_{\lambda/\mu}\,\left(\dfrac{1}{1+\beta}\right)^{r(\mu/\Tilde{\lambda})}\,\left(\dfrac{\beta}{1+\beta}\right)^{|\lambda/\mu|}\,\prod_{i\in E_{\lambda/\mu}}(1+\beta q^{m_{i}(\lambda)}) \ge 0.
$$

(iii) Let $\gamma \ge 0$. 
Consider again the expansion 
$$\ff_{\lambda} = \sum_{\mu} \tilde{a}_{\lambda, \mu}\, W_{\mu}$$ and its $\theta_{Pl, \gamma}$ specialization given by 
$$
\hat\theta_{Pl, \gamma} : \ff_{\lambda} \longmapsto \sum_{\mu} \tilde{a}_{\lambda, \mu}\, \theta_{Pl, \gamma}(W_{\mu}) 
\in \mathbb{R}[[\gamma]], 
$$
We are going to show that the series $\sum_{\mu} \tilde{a}_{\lambda, \mu}\, \theta_{Pl, \gamma}(W_{\mu})$ converges absolutely for all $\gamma \in \mathbb{R}$.

Instead, we are going to show equivalently that the same is true for a positive version of inhomogeneous $q$-Whittaker functions $\overline{\ff}_{\lambda}(\mathbf{x}) := (-1)^{|\lambda|}\ff_{\lambda}(-\mathbf{x})$, for which we have 
$$
\hat\theta_{Pl, \gamma} : \overline{\ff}_{\lambda} \longmapsto 
 \sum_{\mu} \tilde{b}_{\lambda, \mu}\, \theta_{Pl, \gamma}(W_{\mu}) 
 \in \mathbb{R}_{\ge 0}[[\gamma]], \quad \tilde{b}_{\lambda, \mu} := (-1)^{|\lambda|-|\mu|}\tilde{a}_{\lambda, \mu} \ge 0.
$$
Let us first note that 
$$
\hat\theta_{Pl, \gamma} : \overline{\ff}_{1} \longmapsto \sum_{k \ge 1}  \theta_{Pl, \gamma}(W_{1^k}) = \sum_{k \ge 1}  \frac{\gamma^k}{k!} = e^{\gamma} - 1. 
$$
By iteratively applying Proposition~\ref{prop:Inwhit_Pieri} it follows that $\overline{\ff}_{1}^{|\lambda|}$ is $\overline{\ff}$-positive function and let $c_{\lambda} > 0$ be its coefficient at $\overline{\ff}_{\lambda}$. 
Consider the positive expansions 
$$
\overline{\ff}_{1}^{|\lambda|} = \sum_{\nu} c_{\nu}\, \overline{\ff}_{\nu} = \sum_{\mu} a_{\mu} W_{\mu}, \quad c_{\nu}, a_{\mu} \ge 0.
$$
Note that $a_{\mu} \ge c_{\lambda}\, \tilde{b}_{\lambda, \mu}$ since the functions $\overline{\ff}_{\nu}$ are $W$-positive. 
We then have 
$$
\hat\theta_{Pl, \gamma} : c_{\lambda}\, \overline{\ff}_{\lambda} \longmapsto \sum_{\mu} c_{\lambda}\, \tilde{b}_{\lambda,\mu}\, \theta_{Pl, \gamma}(W_{\mu}) \le \sum_{\mu} a_{\mu}\, \theta_{Pl, \gamma}(W_{\mu}) = \hat\theta_{Pl, \gamma}(\overline{\ff}_1^{|\lambda|}) = (e^{\gamma} - 1)^{|\lambda|},
$$
which gives that $\overline{\ff}_{\lambda}$ maps to absolutely convergent series for all $\gamma \in \mathbb{R}$. Therefore, 
$$
\hat\theta_{Pl, \gamma} : \ff_{\lambda} \longmapsto \sum_{\mu} \tilde{a}_{\lambda, \mu}\, \theta_{Pl, \gamma}(W_{\mu})
$$
converges absolutely for all $\gamma \in \mathbb{R}$ as well. 

As we have established convergence, using simple variable substitution and part (i) we can then write that 
$$
\hat\theta_{Pl, \gamma}(\ff_{\lambda/\mu})  
= \lim_{N\mapsto \infty} \ff_{\lambda/\mu}(\underbrace{\hat\theta_{\gamma/N}, \dots, \hat\theta_{\gamma/N}}_{N \text{ times}}) 
= \lim_{N\mapsto \infty} \ff_{\lambda/\mu}(\underbrace{\gamma/N, \dots,\gamma/N}_{N \text{ times}}) \ge 0,
$$
which also shows positivity. 
\end{proof}

We then have the following structural description of $\ff$-positive specializations. 

\begin{thm}\label{thm:fact}
We have: $\varphi : \Gamma^q \to \mathbb{R}$ is $\mathbb{F}$-positive specialization such that     
$
\sum_{n} \varphi_n < \infty,
$
if and only if $\varphi$ can be factorized via $\ff$-positive generators as  
$$
\varphi = \hat\theta_{Pl, \gamma} \cup \hat\theta_{\alpha_1} \cup \hat\theta_{\alpha_2} \cup \cdots \cup \hat\theta'_{\beta_1} \cup \hat\theta'_{\beta_2} \cup \cdots 
$$
for parameters $1 \ge \alpha_1 \ge \alpha_2 \ge \cdots \ge 0$, $\beta_1 \ge \beta_2 \ge \cdots \ge 0$, $\gamma \ge 0$ such that $\sum_n \varphi_n = \sum_{n} \alpha_n + (1-q)\sum_n\beta_n + \gamma < \infty$. 
\end{thm}

\begin{proof}
Follows from Theorem~\ref{thm:F-positive_specializations} and Lemma~\ref{lem:genf}. 
\end{proof}

\begin{lem}\label{lem:fi}
Let $\varphi : \Gamma^q \to \mathbb{R}$ be $\ff$-positive specialization. Then we have 
$$
1 \ge \varphi_1 \ge \varphi_2 \ge \cdots \ge 0.
$$
\end{lem}
\begin{proof}
    Since for all $n \ge 0$ the following identity holds
    $$
\ff_{1^n} - \ff_{1^{n+1}}  = \ff_{1^{n+1}/1},
    $$
    we have $\varphi(\ff_{1^n}) - \varphi(\ff_{1^{n+1}}) = \varphi(\ff_{1^{n+1}/1}) \ge 0$ as needed. 
\end{proof}

\subsection{Some useful properties of $q$-Whittaker polynomials}
\begin{prop}[{\cite[Corollary~5.2]{GWZJ2025}}]
For any partition $\lambda$, consider the (infinite) expansion
\[
  W_{\lambda}
  \;=\;
  \sum_{\mu}\,a_{\lambda,\mu}(q)\,
  \mathbb{F}_{\mu}. 
\]
Then $a_{\lambda,\mu}(q) \in \mathbb{N}[q]$ and Moreover $a_{\lambda,\mu}(q)\geq 0$ for all $|q|<1$. In other words, the $q$-Whittaker polynomials $W_{\lambda}$ expand positively in terms of the inhomogeneous $q$-Whittaker polynomials $\mathbb{F}_{\mu}$.
\end{prop}

\begin{proof}
We briefly recall the formula for the coefficients $a_{\lambda,\mu}(q)$ from~\cite[Corollary~5.2]{GWZJ2025}.  
These coefficients arise as the partition function of a lattice model where the vertex weights as just $q$-Binomial coefficients. Hence, we can conclude that the coefficients $a_{\lambda,\mu}$ are a finite sum of products of $q$-binomial coefficients. In particular, this shows that $a_{\lambda,\mu}(q)\in\mathbb{N}[q]$. Furthermore, as $q$-binomial coefficients are ratios of $q$-Pochhammer symbols, it follows that $a_{\lambda,\mu}(q)\ge 0$ for all $|q|<1$.
\end{proof}

\begin{prop}[{\cite[Corollary~5.4]{GWZJ2025}}] For any  partition $\lambda$, consider the (infinite) expansion
\begin{equation}
\mathbb{F}_{\lambda}
\,=\,\sum_{\mu}\,b_{\lambda,\mu}(q)\,{W}_{\mu}.
\end{equation}  
Then $(-1)^{|\lambda|-|\mu|}b_{\lambda,\mu}(q)\in \mathbb{N}[q]$.
\end{prop}

\begin{prop}[{\cite[p.~46]{spin-BW2021}}]\label{prop:whit_Pieri}
The $q$-Whittaker polynomials satisfy the following product rule:
\begin{equation}
\dfrac{1}{(q;q)_{i}}\,W_{i} \,W_{\nu} \\
   =\; \sum_{\lambda}\, d_{\lambda/\nu}\,W_{\lambda},
\end{equation}
where the summation is over partitions $\nu\prec\lambda$ such that $|\lambda/\nu|=i$ and we have

\[
d_{\lambda/\nu}\,=\,\prod_{j\geq 1}\,\dfrac{(q;q)_{\nu_{j}-\nu_{j+1}}}{(q;q)_{\lambda_{j}-\nu_{j}}\,(q;q)_{\nu_{j}-\lambda_{j+1}}}.
\]
\end{prop}
\

\begin{prop}\label{prop:w1}
For $|q| < 1$ and positive integers $n$ we have the following $W$-positive expansions: 
\begin{equation}\label{eq:W_1^n}
W_1^n = \sum_{\lambda \vdash n} f_{\lambda} W_{\lambda}, 
\end{equation}
where $f_{\lambda} > 0$ for each partition $\lambda \vdash n$.
\end{prop}
\begin{proof}
We first observe that $d_{\lambda/\nu}>0$ whenever $-1<q<1$. 
The claim then follows by repeated application of Proposition~\ref{prop:whit_Pieri}.
\end{proof}

\begin{prop}\label{prop:wf}
    The following expansions hold: 
    \begin{align}
    W_{1^n}\, &=\, \ff_{1^n} + \binom{n}{1} \ff_{1^{n+1}} + \binom{n+1}{2} \ff_{1^{n+2}} +\binom{n+2}{3}\,\ff_{1^{n+3}}+ \cdots, \quad n = 1,2, \ldots.\\
        \ff_{1^{n}}\,&=\,W_{1^{n}}-\binom{n}{1}W_{1^{n+1}}+\binom{n+1}{2}\,W_{1^{n+2}}-\binom{n+2}{3}\,W_{1^{n+3}}+\cdots, \quad n = 1,2, \ldots.
    \end{align}
\end{prop}
\begin{proof}
We observe from the branching formulas for $W_{\lambda}$ and $\ff_{\lambda}$ that the polynomials $W_{1^{n}}$ and $\ff_{1^{n}}$ are independent of the parameter $q$.  
Together with the fact that, $W_{1^{n}}$ and $\ff_{1^{n}}$ reduce respectively to the Schur and Grothendieck polynomials  at $q=0$, this implies that
\[
W_{1^{n}} = s_{1^{n}},
\qquad
\ff_{1^{n}} = G_{1^{n}}.
\]
Then the expansions in the claim follow from Lenart’s formula for expressing Schur polynomials in terms of Grothendieck polynomials, together with the inverse expansion \cite[Theorem~2.7]{Lenart2000}; see also Definition 3.4 and the equation below it in \cite{positive:Yel20}.
\end{proof}

\subsection{Completing the proof}
\begin{proof}[Proof of Theorem~\ref{thm:F-positive_specializations}]
(i) 
Since $\phi$ is $W$-positive we have 
$$
\phi = \theta_{Pl, \gamma} \cup \theta_{\alpha_1} \cup \theta_{\alpha_2} \cup \cdots \cup \theta'_{\beta_1} \cup \theta'_{\beta_2} \cup \cdots.
$$
By Lemma~\ref{lem:genf} and taking unions of $\ff$-positive specializations we conclude that 
$$
\varphi = \hat\phi = \hat\theta_{Pl, \gamma} \cup \hat\theta_{\alpha_1} \cup \hat\theta_{\alpha_2} \cup \cdots \cup \hat\theta'_{\beta_1} \cup \hat\theta'_{\beta_2} \cup \cdots
$$
is $\ff$-positive, and for this specialization we have $\sum_{n} \varphi_n = \sum_n \alpha_n + (1-q)\sum_n\beta_n + \gamma < \infty.$

(ii) We are going to show that more generally the ring homomorphism $\phi : \Lambda \to \mathbb{R}$ given on the basis elements $\{W_{\lambda}\}$ by 
$$
\phi(W_{\lambda}) = \sum_{\mu} a_{\lambda, \mu}(q)\, \varphi(\ff_{\mu}) < \infty
$$
is well-defined $W$-positive specialization w.r.t. to $\ff$-expansion of $W_{\lambda}$. First notice that the element $\phi(W_1) = \sum_n \varphi(F_{1^n}) < \infty$ converges. From Proposition~\ref{prop:w1} it follows that for each partition $\mu \vdash m$ we have 
$$f_{\mu}\,\phi(W_{\mu}) \le \sum_{\lambda \vdash m} f_{\lambda} \phi(W_{\lambda}) = \phi(W_1^m) = \phi(W_1)^m < \infty$$ 
converges as well. Since $a_{\lambda, \mu}(q) \ge 0$ we have $\phi(W_{\lambda})\ge 0$ and hence $\phi$ is $W$-positive. Since $\Lambda$ is a polynomial ring generated by $W_{1^n} = e_n$, the specialization $\phi$ can be defined on these elements using the expansion from Proposition~\ref{prop:wf}: 
$$
\phi(W_{1^n}) = \varphi(\ff_{1^n}) + \binom{n}{1} \varphi(\ff_{1^{n+1}}) + \binom{n+1}{2} \varphi(\ff_{1^{n+2}}) + \cdots < \infty, \quad  n = 1,2,\ldots.  
$$
As $\phi$ is defined from $\ff$-expansion of $W_{\lambda}$, it can be seen that $\phi$ also extends back to $\varphi$. 

Since $\phi$ is $W$-positive it can be factorized as 
$$\phi = \theta_{Pl, \gamma} \cup \theta_{\alpha_1} \cup \theta_{\alpha_2} \cup \cdots \cup \theta'_{\beta_1} \cup \theta'_{\beta_2} \cup \cdots$$
for parameters $\alpha_1 \ge \alpha_2 \ge \cdots \ge 0$, $\beta_1 \ge \beta_2 \ge \cdots \ge 0$, $\gamma \ge 0$. Therefore, we also have 
$$\varphi = \hat\theta_{Pl, \gamma} \cup \hat\theta_{\alpha_1} \cup \hat\theta_{\alpha_2} \cup \cdots \cup \hat\theta'_{\beta_1} \cup \hat\theta'_{\beta_2} \cup \cdots.$$
Now we need to show that $\alpha_1 \le 1$. Since by Lemma~\ref{lem:fi} we have $\varphi(\ff_{1^n}) \ge \varphi(\ff_{1^{n+1}})$ for all $n \ge 0$ we also have $\phi(W_{1^n}) \ge \phi(W_{1^{n+1}})$. Hence, the generating series $\sum_n \phi(W_{1^n}) z^n$ has convergence radius at least $1$ and its poles (see Corollary~\ref{cor:wgen}) satisfy $1/\alpha_i \ge 1$ for all $i$. 
\end{proof}

\begin{rmk}
While we showed that $\ff$-positive specializations $\varphi$ such that $\sum_n \varphi_n < \infty$ are in bijection with a subset of $W$-positive specializations, there is also $\ff$-positive specialization with divergent series $\sum_n \varphi_n$: the specialization $\varphi : \ff_{\lambda} \mapsto 1$ for all partitions $\lambda$, which can be implemented by taking all variables $x_i \mapsto 1$ and for which $\varphi_n = 1$ for all $n$. (Note that this fails to be specialization of $\Lambda$.) It can be shown that if $\varphi_1 < 1 - q$ then $\sum_n \varphi_n < \infty$ is always convergent, and it is unclear if there are other $\ff$-positive specializations with divergent series $\sum_n \varphi_n$. 
\end{rmk}

\section{Related probability distributions}
In this section, we show that $\ff$-positive specializations of the ring $\Gamma^q$ give rise to some interesting probability distributions on integer partitions. 

\begin{thm}\label{thm:prob}
Let $q \in (0,1)$ and $\varphi : \Gamma^q \to \mathbb{R}$ be $\mathbb{F}$-positive specialization such that $\varphi_1 < 1$ and $\sum_n \varphi_n < \infty$. Then for any fixed partition $\mu$ there is a well-defined probability distribution on the set of partitions  $\lambda \supseteq \mu$ given by 
$$\mathbb{P}_{\varphi, \mu}(\lambda) := \frac{c_{\lambda, \mu}\, \varphi(\mathbb{F}_{\lambda/\mu})}{Z_{\varphi}}, \qquad c_{\lambda, \mu} := \frac{\prod_{i=1}^{\ell(\lambda')} {(q;q)_{m_{i}(\lambda')}}}{\prod_{i=1}^{\ell(\mu')} {(q;q)_{m_{i}(\mu')}}}
$$ 
and 
normalization constant $Z_{\varphi}$ depending on $\varphi$. 
\end{thm}

The proof of this result is based on factorization Theorem~\ref{thm:fact} and the following Littlewood-type identity for inhomogeneous $q$-Whittaker polynomials.  

\begin{prop}[Littlewood-type identity]
\label{cor:Littlewood_inwhit}
Let $\mu$ be a fixed partition. Then the inhomogeneous $q$-Whittaker polynomials satisfy the following summation identity (assuming that all the variables are in the unit disc):
\begin{equation}\label{eq:spinL_Cauchy_qwhittaker}
\,\sum_{\lambda}\,\dfrac{b_{\mu'}(q)}{b_{\lambda'}(q)}\,\ff_{\lambda/\mu}(x_{1}, x_2,\ldots)\, 
\,=\,\prod^{\infty}_{i=1}\,\dfrac{1}{(x_{i};q)_{\infty}}\,
\end{equation}
where $b_{\lambda}(q) = \prod_{i=1}^{\ell(\lambda)} {(q;q)_{m_{i}(\lambda)}}$ and the sum on the left is over all partitions that contain $\mu$.
\end{prop}
\begin{proof}
First, we recall a basic fact and record an observation concerning the general functions 
$\mathfrak{G}^{\buv}_{\lambda}$ introduced in~\cite{GWZJ2025}. 
These functions contain $\ff_{\lambda}$ as the specialization 
$\mathbf{u}=(u_{1},u_{2},\dots)=(1,1,\dots)$ and $\mathbf{v}=(v_{1},v_{2},\dots)=(0,0,\dots)$. 
Moreover, in one variable $y$ and at $\mathbf{u}=(u_{1},u_{2},\dots)=(0,0,\dots)$ and $\mathbf{v}=(v_{1},v_{2},\dots)=(1,1,\dots)$, we have the following property:
\[
\mathfrak{G}_{\lambda}^{(\mathbf{0},\mathbf{1})}(y)
= y^{\,\#\{\text{columns of }\lambda\}}.
\]

\
The claim then follows immediately from taking $m=1$, $y=1$ and $\mu=\varnothing$, and specializing 
$\mathbf{v}=(v_{1},v_{2},\dots)=(1,1,\dots)$ and 
$\mathbf{u}=(u_{1},u_{2},\dots)=(0,0,\dots)$ in Theorem~4.13 from~\cite{GWZJ2025}.
\end{proof}

\begin{proof}[Proof of Theorem~\ref{thm:prob}]
We are going to show that this follows from Theorem~\ref{thm:fact} and Proposition~\ref{cor:Littlewood_inwhit}.  

Since $\varphi$ is $\ff$-positive we can factorize it as 
$$\varphi = \hat\theta_{Pl, \gamma} \cup \hat\theta_{\alpha_1} \cup \hat\theta_{\alpha_2} \cup \cdots \cup \hat\theta'_{\beta_1} \cup \hat\theta'_{\beta_2} \cup \cdots$$
for parameters $1 > \alpha_1 \ge \alpha_2 \ge \cdots \ge 0$, $\beta_1 \ge \beta_2 \ge \cdots \ge 0$, $\gamma \ge 0$. Note that here we must have $\alpha_1 < 1$ since $\varphi_1 < 1$. 

For a scalar $t \in (0,1)$, let $\varphi^t$ be a rescaled specialization $\varphi$ given by the parameters $(t\alpha_i), (t\beta_i), t\gamma$. Note that $\varphi^t$ is also $\ff$-positive and satisfies $\varphi^t_1 < 1$.  
Note that under these specializations for each $k \ge 0$ we have 
$$
\varphi^{q^k} : \prod_i (1 - q^k x_i) \longmapsto 1 - \varphi^{q^k}_1 \in (0,1].
$$
Then we have 
$$
\varphi : \prod_{i}^{} (x_i; q)_{\infty} = \prod_{k \ge 0} \prod_i (1 - q^k x_i) \longmapsto \prod_{k \ge 0} (1 - \varphi^{q^k}_1) =: Z_{\varphi}^{-1} 
$$
Applying this specialization into the above Littlewood-type identity we obtain that 
$$
\sum_{\lambda} \frac{1}{Z_{\varphi}} \frac{b_{\mu'}(q)}{b_{\lambda'}(q)} \varphi(\ff_{\lambda/\mu}) = 1,
$$
which gives the needed. 
\end{proof}

\begin{rmk}[Plancherel specialization]
For the Plancherel specialization $\varphi = \theta_{Pl, \gamma}$ we have the normalization constant given explicitly by 
$$
Z_{\varphi} = \prod_{k \ge 0} e^{\gamma q^k}  = e^{\gamma q/(1-q)},
$$
and 
$$
\varphi(\ff_{\lambda/\mu}) = \sum_m \mathbb{E}_{\lambda/\mu}(m)\frac{\gamma^m}{m!},
$$
where $\mathbb{E}_{\lambda/\mu}(m)$ is the coefficient at $[x_1 \cdots x_m]$ at monomial expansion of $\ff_{\lambda/\mu}$.
\end{rmk}

\section{Some open questions}

\subsection{Geometric meaning and structure of the ring $\Gamma^q$}
It is known that for $q=0$ the ring $\Gamma^0$ with the basis of stable Grothendieck polynomials is related to the $K$-theory of Grassmannians \cite{Gcom:B2002}. Is there any geometric meaning of the ring $\Gamma^q$ for general $q$?

What is a combinatorial meaning to the structure constants $c^{\lambda}_{\mu \nu}$ of the ring $\Gamma^q$, which generalize Littlewood--Richardson coefficients of stable Grothendieck polynomials? Such an interpretation will also give an alternative proof of our ring theorem for $\Gamma^q$. 

What is the structure of the ring $\Gamma^q$? For stable Grothendieck polynomials, Buch \cite{Gcom:B2002} conjectured that polynomials $\{ G_{\lambda}\}$ can be generated by the elements $\{ G_{(n)}, G_{1^n}\}$ and $(1 - G_1)^{-1}$. Is it true that similar property can be generalized for inhomogeneous $q$-Whittaker functions?

\subsection{Positive specializations of positive inhomogeneous $q$-Whittaker polynomials and inhomogeneous Hall--Littlewood polynomials}
It would also be interesting to study positive specializations for positive versions of inhomogeneous $q$-Whittaker polynomials given by $\overline{\ff}_{\lambda}(\mathbf{x}) = (-1)^{|\lambda|} \ff_{\lambda}(-\mathbf{x})$ (which have positive structure constants), as well as for dual functions, inhomogeneous Hall-Littlewood polynomials $j_{\lambda}$ which form a basis of the ring $\Lambda$. It is relatively straightforward to find a set of positive specializations for these functions, but the difficult part would be to prove its completeness.  

\

\subsection*{Acknowledgements}
This project was initiated in Spring 2024 while the authors were visiting the program ``Geometry, Statistical Mechanics, and Integrability" at the Institute for Pure and Applied Mathematics (IPAM), and we thank IPAM for great collaborative environment. A.G. would like to thank Sasha Garbali, Michael Wheeler, and Paul Zinn-Justin for their advice and guidance over the years. A.G. gratefully acknowledges Michael Wheeler for travel support and for generously sharing his knowledge of symmetric functions and lattice models. A.G. also thanks Kazakh-British Technical University for its hospitality during the visit, where part of this work was carried out.

\

\bibliographystyle{plain}
\bibliography{references}{}
\end{document}